\documentclass[12pt]{amsart}
\usepackage{fullpage}
\usepackage[utf8]{inputenc}
\usepackage[english]{babel}
\usepackage{mathtools}
\usepackage{amssymb}
\usepackage{amsthm}
\usepackage{mathrsfs}
\usepackage{nicefrac}
\usepackage{upgreek}
\usepackage[T1]{fontenc}
\usepackage{graphicx}

\usepackage[all]{xy}
\usepackage{graphics}
\usepackage{color}
\usepackage[T1]{fontenc}
\usepackage{parskip}

\usepackage{enumitem}

\makeatletter
\newtheorem*{rep@theorem}{\rep@title}
\newcommand{\newreptheorem}[2]{%
\newenvironment{rep#1}[1]{%
 \def\rep@title{#2 \ref{##1}}%
 \begin{rep@theorem}}%
 {\end{rep@theorem}}}
\makeatother

\newcommand{\pre}[2]{\langle #1 | #2\rangle}
\newcommand{\dl}{\langle\langle}
\newcommand{\dr}{\rangle\rangle}
\newcommand{\card}[1]{|#1|}

\newcommand{\C}[1]{{\mathcal #1}}

\newtheorem{Theorem}{Theorem}[section]
\newreptheorem{Theorem}{Theorem}
\newtheorem{Lemma}[Theorem]{Lemma}
\newreptheorem{Lemma}{Lemma}
\newtheorem{Claim}{Claim}[Theorem]
\newtheorem{Corollary}[Theorem]{Corollary}
\newreptheorem{Corollary}{Corollary}
\newtheorem{Proposition}[Theorem]{Proposition}
\newreptheorem{Proposition}{Proposition}

\theoremstyle{definition}\newtheorem{mydef}[Theorem]{Definition}
\theoremstyle{remark}\newtheorem{remark}[Theorem]{Remark}
\theoremstyle{definition}\newtheorem{case}{Case}
\makeatletter\@addtoreset{case}{example}\makeatother
\theoremstyle{definition}

\begin{document}

\title{Strong boundedness of simply connected split Chevalley groups defined over rings}

\author{Alexander A. Trost}
\address{University of Aberdeen}
\email{r01aat17@abdn.ac.uk}

\begin{abstract}
This paper is concerned with the diameter of certain word norms on S-arithmetic split Chevalley groups. Such groups are well known to be boundedly generated by root elements. We prove that word metrics given by conjugacy classes on S-arithmetic split Chevalley groups have an upper bound only depending on the number of conjugacy classes. This property, called \textit{strong boundedness}, was introduced by K\k{e}dra, Libmann and Martin in \cite{KLM} and proven for ${\rm SL}_n(R)$, assuming 
$R$ is a principal ideal domain and $n\geq 3.$ We also provide examples of normal generating sets for S-arithmetic split Chevalley groups proving our bounds are sharp in an appropriate sense and give a complete account of the existence of small normally generating sets of ${\rm Sp}_4(R)$ and $G_2(R)$.
For instance, we prove that ${\rm Sp}_4(\mathbb{Z}[\frac{1+\sqrt{-7}}{2}])$ cannot be generated by a single conjugacy class.
\end{abstract}

\maketitle

\section{Introduction}

The main concept we study in this paper is strong boundedness of groups. For split Chevalley groups it arises from a couple of different sources. Most importantly, it is related to bounded generation of groups and to the diameter of the Cayley graph of the group with respect to certain infinite sets of generators of said group. 

Firstly, a group $G$ is called \textit{boundedly generated} by a set $S\subset G$, if there is a natural number $N:=N(S)$ such that $G=(SS^{-1})^N.$ 
Bounded generation of split Chevalley groups has been widely studied. For example, for S-arithmetic, split Chevalley groups, Tavgen proved in \cite{MR1044049} that all split Chevalley groups of rank at least $2$ defined using S-algebraic integers, have bounded generation with respect to root elements.
We define precisely both split Chevalley groups $G(\Phi,R)$ and their root elements in Section~\ref{sec_basic_notions}, but for the purpose of this introduction, the reader can think about classical matrix groups like ${\rm SL}_n$ and ${\rm Sp}_{2n}$.  
Furthermore, Morris \cite{MR2357719} has extended Tavgen's result to localizations of orders in rings of algebraic integers in the case of the elementary subgroup 
of ${\rm SL}_n$  and Morgan, Rapinchuk, Sury \cite{MR3892969} established bounded generation by root elements, even in the case of ${\rm SL}_2,$ if the underlying ring of S-algebraic integers has infinitely many units. 

Secondly, K\k{e}dra, Libman, Martin \cite{KLM} considered word norms for generating sets consisting of finitely many conjugacy classes. Namely,for a subset $S$ normally generating $G$, the word norm $\|g\|_S$ for $g\in G$ is the smallest number of conjugates of elements of $S\cup S^{-1}$ needed to write $g\in G.$ The diameter $\|G\|_S$, if it is finite, depends on the normally generating set $S$. However, the notion of \textit{strong boundedness} states that 
$\|G\|_S$ has at least an upper bound only depending on the cardinality $|S|.$ 

The first example of this behaviour is presented in the next theorem. In it $\|\cdot\|_{EL(n)}$ denotes the word metric of ${\rm SL}_n$ with respect to the generating set of elementary matrices-that is unipotent matrices with at most $1$ non-zero off-diagonal entry. 

\begin{Theorem}\cite[Theorem~6.1]{KLM}
\label{KLM_thm}
Let $R$ be a principal ideal domain and let ${\rm SL}_n(R)$ be boundedly generated by elementary matrices for $n\geq 3$ with the diameter 
$\|{\rm SL}_n(R)\|_{EL(n)}$ satisfying $\|{\rm SL}_n(R)\|_{EL(n)}\leq C_n$ for some $C_n\in\mathbb{N}.$ Then ${\rm SL}_n(R)$ is normally generated by the single element $E_{1,n}(1)$ and
\begin{enumerate}
\item{for all finite, normally generating subsets $S$ of $G$, it holds $\|{\rm SL}_n(R)\|_S\leq C_n(4n+4)|S|.$
}
\item{if $R$ has infinitely many maximal ideals, then for each $k\in\mathbb{N}$ there is a finite, normally generating subset $S_k$ of $G$ with $|S_k|=k$ and
$\|{\rm SL}_n(R)\|_{S_k}\geq k$.}
\end{enumerate}
\end{Theorem}

The proof of this theorem uses extensive matrix calculations and relies heavily on the underlying ring being a principal ideal domain as well as bounded generation by elementary matrices. Bounded generation can be obtained from Tavgen \cite{MR1044049} and so one of the possible applications would be rings of algebraic integers with class number $1.$ However, it is well known that not all rings of algebraic integers are principal ideal domains and the paper \cite{MR1044049} speaks about more general matrix groups aside from ${\rm SL}_n$ and about arbitrary rings of algebraic integers. 

In this paper we prove the following generalization of part (1):

\begin{repTheorem}{exceptional Chevalley}
Let $\Phi$ be an irreducible root system of rank at least $2$ and let $R$ be a commutative ring with $1$. Additionally, let $G(\Phi,R)$ be boundedly generated by root elements and if $\Phi=B_2$ or $G_2$, then we further assume $(R:2R)<\infty.$
Then there is a constant $C(\Phi,R)\in\mathbb{N}$ such that for all finite, normally generating subset $S$ of $G$, it holds
\begin{equation*}
\|G(\Phi,R)\|_S\leq C(\Phi,R)|S|.
\end{equation*}
\end{repTheorem}

\begin{remark}
Root elements are natural generalizations of the elementary matrices in ${\rm SL}_n$. Such root elements are usually denoted by $\varepsilon_{\chi}(x)$ with varying $\chi\in\Phi$ and $x\in R$. Most notably 
\begin{equation*}
\varepsilon_{\chi}(x_1+x_2)=\varepsilon_{\chi}(x_1)\varepsilon_{\chi}(x_2)
\end{equation*}
holds for all $x_1,x_2\in R.$
\end{remark}

The proof of both Theorem~\ref{exceptional Chevalley} and part (1) of Theorem~\ref{KLM_thm} has the same two step strategy. First, one obtains 
arbitrary root elements as bounded products of conjugates of the finite normally generating set in question and then secondly, one uses bounded generation of the group by root elements to finish. The second step is virtually the same in both cases. However, in the first step instead of using explicit matrix calculations, we use results about the structure of normal subgroups of matrix groups and G\"odel's Compactness Theorem, which enables us to treat more general rings. 
Beyond that there are some features in the rank $2$-cases (more precisely ${\rm Sp}_4$ and $G_2$), which do not occur in the higher rank cases.

Theorem~\ref{exceptional Chevalley} is fairly abstract and can in principle be applied to a lot of different rings. In consequence, we get a couple of corollaries.
First, for rings of S-algebraic integers we obtain Theorem~\ref{alg_numbers_strong_bound}. Second, there is a result for rings of stable range $1$ (Theorem~\ref{stable_range1_strong_boundedness}) and more specifically for semilocal rings (Theorem~\ref{semilocal_uniform}). 

For rings of S-algebraic integers, we also construct finite, normally generating subsets of $G(\Phi,R)$ in Section~\ref{section_lower_bounds}, that give generalizations of Theorem~\ref{KLM_thm}(2). For $\Phi=B_2$ or $G_2$ the situation is more complex: Namely there is a problem with small normally generating sets and we will give a complete account of this. One possible example of this issue is the following:

\begin{repCorollary}{small_generating_sets_Sp4}
Let $R$ be the ring of algebraic integers in the number field $\mathbb{Q}[\sqrt{-7}].$ Then ${\rm Sp}_4(R)$ and $G_2(R)$ are not generated by a single conjugacy class.
\end{repCorollary}

The paper is structured as follows: In Section~\ref{sec_basic_notions}, we define all needed notions like split Chevalley groups, their congruence subgroups and root elements, level ideals and the word norms, which we are studying. In Section~\ref{proof_main}, we state the main technical result and explain how to obtain the main theorem from it. In Section~\ref{proof_fundamental_prop}, we prove this technical result. Both of these sections are split up according to the particular root system $\Phi$ in question, as the arguments are quite different for different $\Phi.$ Section~\ref{Corollaries} speaks about various classes of rings that fulfill the assumptions of Theorem~\ref{exceptional Chevalley} and gives different versions of it. Lastly, in Section~\ref{section_lower_bounds}, we construct explicitly various finite, normal generating sets for the $G(\Phi,R)$ in case of $R$ a ring of S-algebraic integers. In the same section, we also give a complete description of when $G_2(R)$ and ${\rm Sp}_4(R)$ fail to have small normal generating sets for such rings and we make precise, what we mean by \textit{small}.

\section*{Acknowledgments}

I want to thank Bastien Karlhofer for pointing out Proposition~\ref{Vaserstein_decomposition} to me and for always being willing to listen and talk about mathematics. Further, I want to thank Ehud Meir for helpful comments regarding how to write a paper, Ben Martin for being available if I had questions and him and Jarek K\k{e}dra for tirelessly reading several iterations of this paper. This work was funded by Leverhulme Trust Research Project Grant RPG-2017-159.

\section{Basic definitions and notions}
\label{sec_basic_notions}

First, we introduce the basic notions of boundedness and word metrics we study in this paper: 

\begin{mydef}
Let $G$ be a group.
\begin{enumerate}
\item{The notation $A\sim B$ for $A,B\in G$ denotes that $A,B$ are conjugate in $G$. Secondly we define $A^B:=BAB^{-1}$ for $A,B\in G$.}
\item{For $S\subset G$, we define $\dl S\dr$ as the smallest normal subgroup of $G$ containing $S.$}
\item{A subset $S\subset G$ is called a \textit{normally generating set} of $G$, if $\dl S\dr=G$.}
\item{
The group $G$ is called \textit{finitely normally generated}, if a finite normally generating set $S$ exists.}
\item{For $k\in\mathbb{N}$ and $S\subset G$ denote by 
\begin{equation*}
B_S(k):=\bigcup_{1\leq i\leq k}\{x_1\cdots x_i|\ x_j\sim A\text{ or }x_j\sim A^{-1}\text{ for all }j\leq i\text{ and }A\in S\}\cup\{1\}.
\end{equation*}
Further set $B_S(0):=\{1\}.$ If S only contains the single element $A$, then we write $B_A(k)$ instead of $B_{\{A\}}(k)$.}
\item{Define for a set $S\subset G$ the conjugation invariant word norm $\|\cdot\|_S:G\to\mathbb{N}_0\cup\{+\infty\}$ by 
$\|A\|_S:=\min\{k\in\mathbb{N}_0|A\in B_S(k)\}$ for $A\in\dl S\dr$ and by $\|A\|_S:=+\infty$ for $A\notin\dl S\dr.$ The diameter 
$\|G\|_S={\rm diam}(\|\cdot\|_S)$ of $G$ is defined as the minimal $N\in\mathbb{N}$ such that $\|A\|_S\leq N$ for all $A\in G$ or as $\infty$ if there is no such $N$.}
\item{Define for $k\in\mathbb{N}$ the invariant 
\begin{equation*}
\Delta_k(G):=\sup\{{\rm diam}(\|\cdot\|_S)|\ S\subset G\text{ with }\card{S}\leq k,\dl S\dr=G\}\in\mathbb{N}_0\cup\{\rm\infty\}
\end{equation*}
with $\Delta_k(G)$ defined as $-\infty$, if there is no normally generating set $S\subset G$ with $\card{S}\leq k.$ 
}
\item{The group $G$ is called \textit{strongly bounded}, if $\Delta_k(G)$ is finite for all $k\in\mathbb{N}$. 
It is called \textit{uniformly bounded}, if there is a single global bound $L(G)\in\mathbb{N}$ with 
$\Delta_k(G)\leq L(G)$ for all $k\in\mathbb{N}.$}
\end{enumerate}
\end{mydef}

\begin{remark}
\hfill
\begin{enumerate}
\item{Note $\Delta_k(G)\leq\Delta_{k+1}(G)$ for all $k\in\mathbb{N}$.}
\item{
A group $G$ is called \textit{bounded} if $\nu(G)<+\infty$ holds for every conjugation-invariant norm $\nu:G\to\mathbb{R}_{\geq 0}$. For finitely normally generated groups this is equivalent to the existence of a finite normally generating set $S$ such that 
\begin{equation*}
{\rm diam}(\|\cdot\|_S)<\infty.
\end{equation*}
Boundedness properties are not well behaved under passage to finite index subgroups. For example the infinite dihedral group $D_{\infty}$ is bounded, but its finite index subgroup $\mathbb{Z}$ is not.
}
\end{enumerate} 
\end{remark}

\subsection{Simply connected split Chevalley groups}\label{naturality_Chevalley}

To define split Chevalley groups we will first define the Chevalley-Demazure group scheme. We do not prove various statements in the course of this definition. For a more complete description with implicit claims shown please consider \cite{MR1611814} and \cite[Theorem~1, Chapter~1, p.7; Theorem~6(e), Chapter~5, p.38; Lemma~27, Chapter~3, p.~29]{MR3616493}.

Let $G$ be a simply-connected, semi-simple complex Lie group and $T$ a maximal torus in $G$ with associated irreducible root system $\Phi.$ Further, denote by 
$\Pi$ a system of positive, simple roots of $\Phi,$
by $\mathfrak{g}$ the corresponding complex semi-simple Lie-algebra of $G$. The Cartan-subalgebra corresponding to $T$ will be denoted by $\mathfrak{h}$ and the corresponding root spaces in $\mathfrak{g}$ by $E_{\phi}$ for $\phi\in\Phi.$ These choices of Cartan-subalgebra and (simple, positive) roots will be fixed throughout the paper.
The Lie-algebra $\mathfrak{g}$ has a so-called \textit{Chevalley basis} 
\begin{equation*}
\{X_{\phi}\in E_{\phi}\}_{\{\phi\in\Phi\}}\cup\{H_{\phi}\}_{\{\phi\in\Pi\}}
\end{equation*}
such that the structure constants of the Lie Algebra $\mathfrak{g}$ with respect to this basis are all integral. Chevalley-basis are unique up to signs and automorphisms of $\mathfrak{g}$.
\

For each faithful, continuous representation $\rho:G\to GL(V)$ for a complex vector space $V$, there is a lattice $V_{\mathbb{Z}}$ in $V$ with the property:
\begin{equation*}
\frac{d\rho(X_{\phi})^k}{k!}\left(V_{\mathbb{Z}}\right)\subset V_{\mathbb{Z}}\text{ for all }\phi\in\Phi\text{ and }k\geq 0.
\end{equation*}

Fixing a minimal generating set $\{v_1,\dots,v_n\}$ of $V_{\mathbb{Z}}$, then defines functions $t_{ij}:G\to\mathbb{C}$ for all $1\leq i,j\leq n$ by:
\begin{equation*}
\rho(g)(v_i)=\sum_{j=1}^n t_{ij}(g)v_j,
\end{equation*}
because the set $\{v_1,\dots,v_n\}$ also defines a $\mathbb{C}$-basis of $V.$
The functions $t_{ij}$ generate a Hopf algebra called $\mathbb{Z}[G]$ and this defines the Chevalley-Demazure group scheme by
\begin{equation*}
G(\Phi,\cdot): R\mapsto G(\Phi,R):={\rm Hom}_{\mathbb{Z}}(\mathbb{Z}[G],R)
\end{equation*}
with the group structure on $G(\Phi,R)$ given by the Hopf algebra structure on $\mathbb{Z}[G]$ and the induced group homomorphisms
$G(\Phi,R)\to G(\Phi,S)$ obtained by postcomposing with the ring homomorphism $R\to S$. This group scheme $G(\Phi,\cdot)$ does not depend up to isomorphism on the choices of Chevalley basis, faithful representation $\rho$ and lattice $V_{\mathbb{Z}}.$

Further note, that the ring $\mathbb{Z}[y_{ij}]$ is a finitely generated $\mathbb{Z}$-algebra and $\mathbb{Z}$ is noetherian. Hence the polynomial ring in several unknown $\mathbb{Z}[y_{ij}]$ is noetherian and hence there is a finite collection of polynomial functions $P$ in $\mathbb{Z}[y_{ij}]$ such that 
$\mathbb{Z}[y_{ij}]/(P(y_{ij}))\cong\mathbb{Z}[G]$ with the isomorphism given by $y_{ij}\mapsto t_{ij}.$

Using this, one can equivalently define $G(\Phi,R)$ as a subgroup of $GL_n(R)$ by setting:
\begin{equation*}
G(\Phi,R):=\{A\in R^{n\times n}|P(A)=0\}.
\end{equation*}
In this notation, the induced maps $G(\Phi,R)\to G(\Phi,S)$ are obtained by entry-wise application of the ring homomorphism $R\to S.$
We will use mostly this interpretation of $G(\Phi,R)$ in the course of this paper. 

\begin{remark}
In terms of algebraic groups, the group $G(\Phi,R)$ is the group of $R$-points of the $\mathbb{Z}$-defined group scheme $G(\Phi,\cdot).$ 
\end{remark}

\subsection{Root elements}

Next, we will define the previously mentioned root elements of Chevalley groups. For this end, fix a root $\alpha\in\Phi$ and observe that for 
$Z\in\mathbb{C}$ arbitrary the following function is an element of $\rho(G)\subset GL(V):$
\begin{equation*}
\varepsilon_{\alpha}(Z):=\sum_{k=0}^{\infty}\frac{(Zd\rho(X_{\alpha}))^k}{k!}
\end{equation*}
Further $\rho(\varepsilon_{\alpha}(Z))$ in $GL(V)$ has coordinates with respect to the basis $\{v_1,\dots,v_n\}$ that are polynomial functions in $Z$ with coefficients in $\mathbb{Z}.$ This yields a ring homomorphism 
\begin{equation*}
\varepsilon_{\alpha}:\mathbb{Z}[G]\to\mathbb{Z}[Z].
\end{equation*}
By precomposing, this defines another map as follows:
\begin{equation*}
\varepsilon_{\alpha}:\varepsilon_{\alpha}(R):={\rm Hom}_{\mathbb{Z}}(\mathbb{Z}[Z],R)\to {\rm Hom}_{\mathbb{Z}}(\mathbb{Z}[G],R)=G(\Phi,R) 
\end{equation*}
Lastly, the root elements $\varepsilon_{\alpha}(x)\in G(\Phi,R)$ for $x\in R$ are defined as the image of the map 
$x:\mathbb{Z}[Z]\to R,Z\mapsto x$ under the map $\varepsilon_{\alpha}.$ 

The \textit{elementary subgroup} $E(\Phi,R)$ (or $E(R)$ if $\Phi$ is clear from the context) is defined as the subgroup of $G(\Phi,R)$ generated by the elements 
$\varepsilon_{\alpha}(x)$ for $\alpha\in\Phi$ and $x\in R.$ We refer the reader to \cite{MR3616493} for further details regarding root elements. 

Also note the following property:

\begin{mydef}
Let $R$ be a commutative ring with $1$. Then $G(\Phi,R)$ is \textit{boundedly generated by root elements}, if there is a natural number $N\in\mathbb{N}$ and 
roots $\alpha_1,\dots,\alpha_N\in\Phi$ such that for all $A\in G(\Phi,R)$, there are $a_1,\dots,a_N\in R$ (depending on $A$) such that:
\begin{equation*}
A=\prod_{i=1}^N\varepsilon_{\alpha_i}(a_i).
\end{equation*} 
\end{mydef}

The symbols $\varepsilon_{\alpha}(t)$ are additive in $t\in R$, that is $\varepsilon_{\alpha}(t+s)=\varepsilon_{\alpha}(t)\varepsilon_{\alpha}(s)$ holds for all $t,s\in R$ and a couple of commutator formulas expressed in the next lemma, hold. We will use the additivity and the commutator formulas implicitly throughout the paper usually without reference. 

\begin{Lemma}\cite[Proposition~33.2-33.5]{MR0396773}
\label{commutator_relations}
Let $\alpha,\beta\in\Phi$ be roots with $\alpha+\beta\neq 0$ and $a,b\in R$ be given.
\begin{enumerate}
\item{If $\alpha+\beta\notin\Phi$, then $(\varepsilon_{\alpha}(a),\varepsilon_{\beta}(b))=1.$}
\item{If $\alpha,\beta$ are positive, simple roots in a root subsystem of $\Phi$ isomorphic to $A_2$, then\\ 
$(\varepsilon_{\beta}(b),\varepsilon_{\alpha}(a))=\varepsilon_{\alpha+\beta}(\pm ab).$}
\item{If $\alpha,\beta$ are positive, simple roots in a root subsystem of $\Phi$ isomorphic to $B_2$ with $\alpha$ short and $\beta$ long, then
\begin{align*}
&(\varepsilon_{\alpha+\beta}(b),\varepsilon_{\alpha}(a))=\varepsilon_{2\alpha+\beta}(\pm 2ab)\text{ and}\\
&(\varepsilon_{\beta}(b),\varepsilon_{\alpha}(a))=\varepsilon_{\alpha+\beta}(\pm ab)\varepsilon_{2\alpha+\beta}(\pm a^2b).
\end{align*}
}
\item{If $\alpha,\beta$ are positive simple roots in a root system of $\Phi$ isomorphic to $G_2$ with $\alpha$ short and $\beta$ long, then 
\begin{align*}
&(\varepsilon_{\beta}(b),\varepsilon_{\alpha}(a))=\varepsilon_{\alpha+\beta}(\pm ab)\varepsilon_{2\alpha+\beta}(\pm a^2b)\varepsilon_{3\alpha+\beta}(\pm a^3b)
\varepsilon_{3\alpha+2\beta}(\pm a^3b^2),\\
&(\varepsilon_{\alpha+\beta}(b),\varepsilon_{\alpha}(a))=
\varepsilon_{2\alpha+\beta}(\pm 2ab)\varepsilon_{3\alpha+\beta}(\pm 3a^2b)\varepsilon_{3\alpha+2\beta}(\pm 3ab^2),\\
&(\varepsilon_{2\alpha+\beta}(b),\varepsilon_{\alpha}(a))=\varepsilon_{3\alpha+\beta}(\pm 3ab),\\
&(\varepsilon_{3\alpha+\beta}(b),\varepsilon_{\beta}(a))=\varepsilon_{3\alpha+\beta}(\pm ab)\text{ and}\\
&(\varepsilon_{2\alpha+\beta}(b),\varepsilon_{\alpha+\beta}(a))=\varepsilon_{3\alpha+2\beta}(\pm 3ab).
\end{align*}
}
\end{enumerate}
\end{Lemma}

\begin{remark}
Depending on the choice of the Chevalley basis, the signs on the arguments on the right hand side the above commutator formulas might vary. Further, if the chosen basis is not a Chevalley basis the arguments on the right hand side might even contain additional coefficients that are not $1$ or $-1.$ These issues are commonly referred to as \textit{pinning}. The sign problem will not be resolved in this paper, due to the fact that our norms are invariant under taking inverses anyway.
\end{remark}

Before continuing, we will define the Weyl group and diagonal elements in $G(\Phi,R)$:

\begin{mydef}
Let $R$ be a commutative ring with $1$ and let $\Phi$ be a root system. Define for $t\in R^*$ and $\phi\in\Phi$ the elements:
\begin{equation*}
w_{\phi}(t):=\varepsilon_{\phi}(t)\varepsilon_{-\phi}(-t^{-1})\varepsilon_{\phi}(t).
\end{equation*}
We will often write $w_{\phi}:=w_{\phi}(1).$ We also define $h_{\phi}(t):=w_{\phi}(t)w_{\phi}(1)^{-1}$ for $t\in R^*$ and $\phi\in\Phi.$
\end{mydef} 

\begin{remark}
The Weyl group of $G(\Phi,R)$ is a quotient of the group generated by the $w_{\phi}$, but we do not need it for our study.
\end{remark}

Using these Weyl group elements, we can obtain the following lemma:

\begin{Lemma}
Let $R$ be a commutative ring with $1$ and $\Phi$ an irreducible root system. Let $\phi,\alpha \in\Phi$ and $x\in R$ be given.
Then for each normally generating set $S$ of $G(\Phi,R)$ one has 
\begin{equation*}
\|\varepsilon_{\phi}(x)\|_S=\|\varepsilon_{w_{\alpha}(\phi)}(x)\|_S.
\end{equation*}   
Here the element $w_{\alpha}(\phi)$ is defined as $\phi-\langle\phi,\alpha\rangle\alpha.$
\end{Lemma}

\begin{proof}
This is a direct consequence of \cite[Lemma~20(b), Chapter~3, p.~23]{MR3616493}.
\end{proof}

Next, we will define certain congruence subgroups and some other notions that we need later on. 

\begin{mydef}
Let $\Phi$ be an irreducible root system and let $R$ be a commutative ring with $1$ in the following. 
\begin{enumerate}
\item{For each pair $(J,L)$, where $J$ is an ideal in $R$ and $L$ an additive subgroup of $J$, we define the subgroup $E(J,L)$ of $G(\Phi,R)$ as the group generated by all elements of the form $\varepsilon_{\alpha}(x)$ for $\alpha\in\Phi$ short, $x\in J$ and $\varepsilon_{\beta}(y)$ for $\beta\in\Phi$ long, $y\in L$.}
\item{For each such pair $(J,L)$, we define the subgroup $\bar{E}(J,L)$ as the normal closure of $E(J,L)$ in $E(R)$.}
\item{For each such pair $(J,L)$, we define the subgroup $E^*(J,L)$ as follows:
\begin{equation*}
E^*(J,L):=\{A\in G(R,\Phi)|(A,E(R))\subset\bar{E}(J,L)\}.
\end{equation*}
}
\item{For an ideal $J$ in $R$ the map $\pi_J:G(\Phi,R)\to G(\Phi,R/J)$ is the group homomorphism induced by the quotient map $R\to R/J.$}
\item{For $k\in\mathbb{N}_0,S\subset G(\Phi,R)$ and $\chi\in\Phi$ set $\varepsilon(S,\chi,k):=\{r\in R|\varepsilon_{\chi}(r)\in B_S(k)\}$.}
\end{enumerate}
\end{mydef}

\subsection{Central elements of Chevalley groups and level ideals}

Let $G$ be a complex, simply-connected, semi-simple Lie-group with irreducible root system $\Phi$ which is not $B_2$ or $G_2$ and positive, simple roots
$\Pi.$ Then there are representations $\rho_i:G\to GL(V_i)$ for $1\leq i\leq u$ such that for $V:=V_1\oplus\cdots\oplus V_u$ the induced direct sum representation 
$\rho:G\to GL(V)$ is faithful. The precise construction is explained in \cite[Chapter~3, p.~29]{MR3616493}. In case of $\Phi\neq B_2$ or $G_2$, this group is what we refer to as the split Chevalley group $G(\Phi,R).$ Setting further $n_i:=dim_{\mathbb{C}}(V_i)$ for $1\leq i\leq u,$ there is the following description of central elements in $G(\Phi,R)$: 

\begin{Lemma}
\label{central_elements}
Let $R$ be a reduced, commutative ring with $1$ and $\Phi$ an irreducible root system, which is not $B_2$ or $G_2$. Further, let $A\in G(\Phi,R)$ commute with the elements of $E(\Phi,R)$. Then there are $t_1,\dots,t_u\in R^*$ such that $A=(t_1 I_{n_1})\oplus\cdots\oplus(t_u I_{n_u})\in GL(R^{n_1+n_2+\dots+n_u}).$
Furthermore, elements of this form are central in $G(\Phi,R)$.
\end{Lemma}
 
The proof for this lemma is in the Appendix. Presumably this statement holds for general rings $R$, but we were not able to find a reference.
Next, we give the definitions of $G(B_2,R)={\rm Sp}_4(R)$ and $G_2(R)$. While we do not specify the representations $\rho$ used, both 
are still instances of our general definition of $G(\Phi,R)$ in Subsection~\ref{naturality_Chevalley}.

\begin{mydef}
Let $R$ be a commutative ring with $1$ and let 
\begin{equation*}
Sp_4(R):=\{A\in R^{4\times 4}|A^TJA=J\} 
\end{equation*}
be given with 
\begin{equation*}
J=
\begin{pmatrix}
0 & 0 & 1 & 0\\
0 & 0 & 0 & 1\\
-1 & 0 & 0 & 0\\
0 & -1 & 0 & 0\\
\end{pmatrix}
\end{equation*}
The root system $B_2$ has four different positive roots namely, $B_2^+=\{\alpha,\beta,\alpha+\beta,2\alpha+\beta\}$ with 
$\alpha$ short and $\beta$ long and both simple. The corresponding root elements have (subject to the choice of maximal torus) the following form for $t\in R$:
\begin{align*}
&\varepsilon_{\alpha}(t)=I_4+t(e_{12}-e_{43}),\varepsilon_{\alpha+\beta}(t)=I_4+t(e_{14}+e_{23})\\
&\varepsilon_{\beta}(t)=I_4+te_{24},\varepsilon_{2\alpha+\beta}(t)=I_4+te_{13}
\end{align*}
and $\varepsilon_{\phi}(t)=(\varepsilon_{-\phi}(t))^T$ for negative roots $\phi\in B_2.$
\end{mydef}

We could specify an explicit matrix description for $G_2$ as well, but this would be rather lengthy and instead we refer to the description in the appendix of 
\cite{MR1487611}. This appendix gives $G_2$ as a subgroup-scheme of ${\rm GL}_8$.
We will not specify which elements of $G_2\subset{\rm GL}_8$ correspond to root elements in particular, but note the positive roots in $G_2.$ They are 
\begin{equation*}
G_2^+=\{\alpha,\beta,\alpha+\beta,2\alpha+\beta,3\alpha+\beta,3\alpha+2\beta\}
\end{equation*}
with $\alpha$ short and $\beta$ long and both simple. 

Next, we will define various variants of level ideals:

\begin{mydef}
\label{central_elements_def}
Let $R$ be a commutative ring with $1$ and let $A\in G(\Phi,R)$ be given. The \textit{level ideal $l(A)$} is defined as 
\begin{enumerate}
\item{in case $\Phi\neq B_2$ or $G_2$ as the ideal in $R$ generated by the elements $a_{i,j}$ for all $1\leq i\neq j\leq n_1+\cdots+n_u$ and the elements
 $a_{i,i}-a_{n_1+\cdots n_w,n_1+\cdots+ n_w}$ for all $1\leq i<n_1+\cdots+n_u$ and the smallest $w\in\{1,\dots,u\}$ with $i<n_1+\cdots+ n_w.$ 
}
\item{in case $\Phi=B_2$ as $l(A):=\pre{a_{i,j},(a_{i,i}-a_{j,j})}{1\leq i\neq j\leq 4}.$}
\item{in case $\Phi=G_2$ as $l(A):=\pre{a_{i,j},(a_{i,i}-a_{j,j})}{1\leq i\neq j\leq 8}.$}
\end{enumerate}
Furthermore, define the following ideals: If $\Phi=B_2$ define 
\begin{equation*}
l(A)_2:=\pre{a_{i,j}^2,(a_{i,i}-a_{j,j})^2}{1\leq i\neq j\leq 4}
\end{equation*}
and if $\Phi=G_2$ define 
\begin{equation*}
l(A)_3:=\pre{a_{i,j}^3,(a_{i,i}-a_{j,j})^3}{1\leq i\neq j\leq 8}.
\end{equation*} 
\end{mydef} 

\begin{remark}
\hfill
\begin{enumerate}
\item{In case $\Phi=B_2$ or $G_2$, note $l(A)\subset\sqrt{l(A)_2}$ or $l(A)\subset\sqrt{l(A)_3}.$}
\item{The important point in the following discussion is that all of these ideals are finitely generated.}
\end{enumerate}
\end{remark}

\section{Fundamental propositions and the proof of Theorem~\ref{exceptional Chevalley}}
\label{proof_main}

Recall the (following equivalent version of the) main theorem:

\begin{Theorem}
\label{exceptional Chevalley}
Let $\Phi$ be an irreducible root system of rank at least $2$ and let $R$ be a commutative ring with $1$. Additionally, let $G(\Phi,R)$ be boundedly generated by root elements and if $\Phi=B_2$ or $G_2$, we further assume $(R:2R)<\infty.$
Then there is a constant $C(\Phi,R)\in\mathbb{N}$ such that
\begin{equation*}
\Delta_k(G(\Phi,R))\leq C(\Phi,R)k
\end{equation*}
for all $k\in\mathbb{N}.$
\end{Theorem}

The main technical tool to prove the theorem is the following: 

\begin{Theorem}
\label{fundamental_reduction}
Let $\Phi$ be an irreducible root system of rank at least $2$ and let $R$ be a commutative ring with $1$. 
Then there are constants $L(\Phi)\in\mathbb{N}$ (depending only on $\Phi$) such that for $A\in G(\Phi,R)$ it holds that
\begin{enumerate}
\item{for $\Phi\neq B_2,G_2$, there is an ideal $I(A)\subset\varepsilon(A,\chi,L(\Phi))$ for $\chi$ a short root. 
This ideal has the property $l(A)\subset\sqrt{I(A)}.$} 
\item{for $\Phi=B_2$ one has $2l(A)_2\subset\varepsilon(A,\phi,L(\Phi))$ for $\phi\in B_2$ arbitrary.}
\item{for $\Phi=G_2$ one has $l(A)_3\subset\varepsilon(A,\chi,L(\Phi))$ for $\chi=3\alpha+\beta.$}
\end{enumerate}
\end{Theorem}

We further need the two following technical observations. First:

\begin{Lemma}
\label{necessary_cond_conj_gen}
Let $\Phi$ be an irreducible root system of rank at least $2$ and $R$ a commutative ring with $1$ and $G:=G(\Phi,R)$ the corresponding split Chevalley group. 
Further let $S$ be a normally generating set of $G.$ Then $\sum_{A\in S}l(A)=R.$ Also if we define for $T\subset G$ the set 
\begin{equation*}
\Pi(T):=\{ m\text{ proper maximal ideal of $R$}|\ \forall A\in T:l(A)\subset m\}
\end{equation*}
then $\Pi(S)=\emptyset$ is equivalent to $\sum_{A\in S}l(A)=R.$
\end{Lemma}

\begin{proof}
Observe that for $I:=\sum_{A\in S}l(A)$, we have that $\pi_I(A)$ is scalar for all $A\in S$ if $\Phi=B_2$ or $G_2$ and has the form described in 
Lemma~\ref{central_elements} if $\Phi\neq B_2$ or $G_2$. Next, assume there is a proper maximal ideal
$m$ containing $I$. As $S$ normally generates $G$, this implies that $\pi_{m}$ maps $G$ only to diagonal matrices. But $m\neq R$ holds, so we can pick an element $\lambda\notin m$ and then $\varepsilon_{\phi}(\lambda+ m)$ would be diagonal for all $\phi\in\Phi$ and so $\lambda\in m$. This contradiction proves $I=R.$ Lastly the equivalence of $\Pi(S)=\emptyset$ and $\sum_{A\in S}l(A)=R$ is clear. 
\end{proof}

And second:

\begin{Lemma}
\label{congruence_fin}
Let $R$ be a commutative ring with $1$ such that $(R:2R)<\infty$ and such that $G:=G(\Phi,R)$ is boundedly generated by root elements for $\Phi=B_2$ or $G_2$. Further define 
\begin{equation*}
N:=\dl\varepsilon_{\phi}(a)|a\in 2R,\phi\in\Phi\dr.
\end{equation*}
Then the group $G/N$ is finite.
\end{Lemma}

\begin{proof}
We are done, if $N$ has finite index in $G.$ The ideal $2R$ has finite index in $R$
so let $X\subset R$ be a finite set of representatives of $2R$ in $R$. The group $G$ is boundedly generated by root elements and so there is a $n:=n(R)$ and roots
 $\alpha_1,\dots,\alpha_n\in\Phi$ such that for all $A\in G$ there are $r_1,\dots,r_n$ with 
\begin{equation}
A=\prod_{i=1}^n\varepsilon_{\alpha_i}(r_i).
\end{equation}   
Next, choose for each $i$ an element $a_i\in R$ and an $x_i\in X$ such that $r_i=2a_i+x_i.$ Note:
\begin{equation*}
A=\prod_{i=1}^n\varepsilon_{\alpha_i}(r_i)=
\varepsilon_{\alpha_1}(2 a_1)\left[\prod_{i=2}^n\varepsilon_{\alpha_i}(2 a_i)^{(\varepsilon_{\alpha_1}(x_1)\cdots\varepsilon_{\alpha_{i-1}}(x_{i-1}))}\right]\cdot
\left[\prod_{i=1}^n\varepsilon_{\alpha_i}(x_i)\right]
\end{equation*}
Yet the first two factors of $A$ are elements of $N$ and there are only finitely many possibilities for the third factor, so the statement of the lemma follows. 
\end{proof} 

We deal with the three different possibilities for $\Phi$ seperately.

\subsection{The higher-rank case and $A_2$}

\begin{Proposition}
\label{mult_bound}
Let $\Phi$ be any irreducible root system that is not $G_2, B_2$ or $A_1$, $R$ a commutative ring with $1$ and let $S$ be a finite subset of $G:=G(\Phi,R)$ with 
$\Pi(S)=\emptyset$ and let $L(\Phi)$ be thrice as the $L(\Phi)$ in Theorem~\ref{fundamental_reduction}. 
Then we have for all $a\in R$ that $\|\varepsilon_{\phi}(a)\|_S\leq\card{S}L(\Phi),$ where $\phi$ is any root in $\Phi$. 
\end{Proposition}

\begin{proof}
Let $S=\{A_1,\dots,A_n\}$ be given and let $I(A_l)$ be the ideal from Theorem~\ref{fundamental_reduction} for all $l=1,\dots,n.$ Next, consider
the ideal $I:=I(A_1)+\cdots+I(A_n).$ As $I(A_l)\subset~\varepsilon(A_l,\phi,L(\Phi))$ holds for all $l$ and all short roots $\phi$ it is immediately clear that 
$\|\varepsilon_{\phi}(a)\|_S\leq\card{S}L(\Phi)$ holds for all $a\in I.$ Thus it suffices to show that $I=R.$ The radical $\sqrt{I}$ contains the ideal 
$l(A_1)+\cdots+l(A_n)$, which is $R$ by assumption. Hence $I=R$ holds.

This proves the claim of the proposition for short roots. If there are long roots in $\Phi$, then each long root $\phi$ is conjugate to a positive, simple long root in a root subsystem of $\Phi$ isomorphic to $B_2.$ Let $\psi$ be the corresponding short, positive, simple root in this root subsystem. Further according to the short root case, we know 
$\|\varepsilon_{\psi}(a)\|_S\leq\card{S}L(\Phi)$ for all $a\in R$ already. So we obtain 
$\|\varepsilon_{\psi}(1)\|_S,\|\varepsilon_{\psi+\phi}(a)\|_S\leq\card{S}L(\Phi)$ for all $a\in R$ and hence as
\begin{equation*}
(\varepsilon_{\psi}(1),\varepsilon_{\phi}(a))=\varepsilon_{\psi+\phi}(\pm a)\varepsilon_{2\psi+\phi}(\pm a),
\end{equation*}
we obtain $\|\varepsilon_{2\psi+\phi}(a)\|_S\leq 3|S|L(\Phi)$ for all $a\in R$. The root $2\psi+\phi$ is long and so we are done.
\end{proof}

We finish this case of Theorem~\ref{exceptional Chevalley}: Lemma~\ref{necessary_cond_conj_gen} implies $\Pi(S)=\emptyset$ and all root groups in $G(\Phi,R)$ are bounded with respect to $\|\cdot\|_S$ with a bound  linear in $|S|.$ However, $G(\Phi,R)$ is also boundedly generated by root elements and hence we are done.

\subsection{The case of ${\rm Sp}_4$}

\begin{Proposition}
\label{mult_bound_B2}
Let $R$ be a commutative ring with $1$ and let $S\subset {\rm Sp}_4(R)$ be a finite set with $\Pi(S)=\emptyset.$ Let $L(B_2)$ be as given in 
Theorem~\ref{fundamental_reduction}. Then we have for all $a\in 2R$ and for all $\phi\in B_2$ that $\|\varepsilon_{\phi}(a)\|_S\leq\card{S}L(B_2)$.
\end{Proposition}

\begin{proof}
Let $S=\{A_1,\dots,A_k\}$ be given and let $2l(A_l)_2$ be the ideal constructed in Theorem~\ref{fundamental_reduction} for all $l=1,\dots,k.$ Consider
the ideal $I:=l(A_1)_2+\cdots+l(A_k)_2.$ As $2l(A_l)_2\subset\varepsilon(A_l,\phi,L(B_2))$ holds for all $l$ and all $\phi\in B_2$, it is immediately clear that 
$\|\varepsilon_{\phi}(2a)\|_S\leq\card{S}L(B_2)$ holds for all $a\in I.$ Thus it suffices to show that $I=R,$ which is clear because $R=\sum_{A\in S}l(A)$ holds by assumption and by construction of $I$ we have $\sum_{A\in S}l(A)\subset\sqrt{I}$.
\end{proof}

To finish the proof of the theorem, we prove next:

\begin{Proposition}
\label{non-simply-laced-reduction}
Let $R$ be a commutative ring with $1$ such that $(R:2R)<\infty$ and let ${\rm Sp}_4(R)$ be boundedly generated by root elements. 
Also let $S$ be a finite subset of ${\rm Sp}_4(R)$ with $\Pi(S)=\emptyset$ and the property that $S$ maps to a normal generating subset of ${\rm Sp}_4(R)/N$ for $N$ as in Lemma~\ref{congruence_fin} and let $F\subset R$ be finite. Then there is a constant $M(B_2,F,R)$ such that $\|\varepsilon_{\phi}(f)\|_S\leq M(B_2,F)\card{S}$ for all $f\in F$ and all $\phi\in B_2.$ So this holds in particular, if $F$ is a finite set of representatives of $2R$ in $R$. 
\end{Proposition}

\begin{proof}
Without loss of generality $F$ only contains a single element $f$. Let $\phi\in B_2$ be arbitrary and note that the group $G/N$ is finite.
Hence there are only finitely many possible normally generating sets of $G/N.$ Call this set of normally generating sets $E(G/N)$. Next, we define a finite set of subsets of $G$ that map to elements of $E(G/N).$ 
By bounded generation there are roots $\alpha_1,\dots,\alpha_n$ such that each element $A$ of $G$ can be written as 
\begin{equation*}
A=\prod_{i=1}^n\varepsilon_{\alpha_i}(r_i)
\end{equation*}      
for particular elements $r_1,\dots,r_n\in R$ depending on $A$. The ring $R/2R$ is finite by assumption and let $X$ be a set of representatives of $2R$ in $R$. Then consider the set $X'$ of elements of the form
\begin{equation*}
A=\prod_{i=1}^n\varepsilon_{\alpha_i}(x_i)
\end{equation*} 
with all $x_i\in X$. Note that $X'$ is finite and hence the set $E(G):=\{T\subset X'|\pi(T)\in E(G/N)\}$ for $\pi:G\to G/N$ the canonical map, is also finite.
 
The group $G/N$ is finite and so there is an $M:=M(B_2,R)\in\mathbb{N}$ such that for all $T\in E(G)$ we can find elements
$t_1,\dots,t_M\in T\cup T^{-1}\cup\{1\},g_1,\dots,g_M\in G$ (all of them depending on $T$) with
\begin{equation}
\label{quotient_eq}
\pi(\varepsilon_{\phi}(f))=\pi\left(\prod_{i=1}^M g_it_ig_i^{-1}\right).
\end{equation}      
Fix such a choice of elements $t_i,g_i$ for each one of the finitely many elements $T\in E(G)$ and call 
the corresponding element $\prod_{i=1}^M g_it_ig_i^{-1}=:e(T).$ The set $E(G)$ is finite and hence the set 
\begin{equation*}
\{\varepsilon_{\phi}(f)e(T)^{-1}|T\in E(G)\}\subset N
\end{equation*}
is finite as well. Next, we prove two claims: First, we show that $S$ only differs by some small terms (with respect to $\|\cdot\|_S$) from an element in $E(G).$ 
Secondly, we demonstrate how to obtain the proposition by using the fact that there is a finite number of possible error terms 
$\{\varepsilon_{\phi}(f)e(T)^{-1}|T\in E(G)\}$.

\begin{Claim}
\end{Claim}
Let $A$ be an element of $S.$ Then, as in the proof of Lemma~\ref{congruence_fin}, we can pick elements $x_i\in X$ and $a_i\in R$ such that
\begin{equation*}
A=\varepsilon_{\alpha_1}(2 a_1)\left[\prod_{i=2}^n\varepsilon_{\alpha_i}(2 a_i)^{(\varepsilon_{\alpha_1}(x_1)\cdots\varepsilon_{\alpha_{i-1}}(x_{i-1}))}\right]\cdot
\left[\prod_{i=1}^n\varepsilon_{\alpha_i}(x_i)\right].
\end{equation*}

Set $A':=\prod_{i=1}^n\varepsilon_{\alpha_i}(x_i)$ and observe
\begin{align*}
\|AA'^{-1}\|_S&\leq\|\varepsilon_{\alpha_1}(2a_1)\|_S+\sum_{i=2}^n\|(\varepsilon_{\alpha_1}(x_1)\cdots\varepsilon_{\alpha_{i-1}}(x_{i-1}))\varepsilon_{\alpha_i}(2a_i)(\varepsilon_{\alpha_1}(x_1)\cdots\varepsilon_{\alpha_{i-1}}(x_{i-1}))^{-1}\|_S\\
&=\sum_{i=1}^n\|\varepsilon_{\alpha_i}(2a_i)\|_S.
\end{align*}

Yet Proposition~\ref{mult_bound_B2} implies $\|\varepsilon_{\alpha_i}(2a_i)\|_S\leq\card{S}L(B_2)$ for all $i$ and hence we can conclude that
$\|AA'^{-1}\|_S\leq\card{S}nL(B_2)$ and so 
\begin{equation}
\label{gen_ineq2}
\|A'\|_S\leq 1+\card{S}nL(B_2). 
\end{equation}
Next, $S$ is an element of $E(G)$ by assumption and hence $S':=\{A'|A\in S\}$ is an element of $E(G)$ as well.
In the following, we use the abbreviation $L:=L(B_2).$

\begin{Claim}
\end{Claim}

Each element of $N$ is a product of conjugates of root elements of the form $\varepsilon_{\phi}(2a)$ for $a\in R$ and $\phi\in B_2$. 
Thus there is a maximal number of such factors in regards to the elements in the finite subset $\{\varepsilon_{\phi}(f)e(T)^{-1}|T\in E(G)\}$ of $N$. If we call this maximal number of factors $V:=V(B_2,R)$, then we obtain by applying Proposition~\ref{mult_bound_B2} that $\|\varepsilon_{\phi}(f)e(T)^{-1}\|_S\leq VL\card{S}$
holds for all $T\in E(G)$. This implies further  
\begin{equation}
\label{gen_ineq}
\|\varepsilon_{\phi}(f)\|_S\leq VL\card{S}+\|e(T)\|_S=VL\card{S}+\|\prod_{i=1}^M g_it_ig_i^{-1}\|_S\leq VL\card{S}+M\max\{\|t\|_S|\ t\in T\}.
\end{equation}

Evaluating (\ref{gen_ineq}) for the particular element $S'\in E(G)$ and applying (\ref{gen_ineq2}) yields 
\begin{equation*}
\|\varepsilon_{\phi}(f)\|_S\leq VL\card{S}+M\max\{\|A'\|_S|\ A'\in S'\}\leq VL\card{S}+M(1+\card{S}nL)=(VL+nLM)\card{S}+M.
\end{equation*}
This finishes the proof.
\end{proof}

\begin{remark}
\hfill
\begin{enumerate}
\item{
There is a second possible proof in a special case using \cite[Theorem]{Gal-Kedra-Trost}. This theorem states that if $R$ is a ring of S-algebraic integers (see definition~\ref{S-algebraic_numbers_def}) and $\Phi$ is an irreducible root system that is not $A_1$, then every finite index subgroup of $G(\Phi,R)$ is bounded. 
The group $N$ has finite index in ${\rm Sp}_4(R)$ and hence by \cite[Theorem]{Gal-Kedra-Trost} it is bounded and normally generated by elements of the form $\varepsilon_{\phi}(2a)$ for $\phi\in B_2$ and $a\in R.$ Using this, one can give a different albeit still very similar proof of the proposition.
A similar argument would yield a generalization of Theorem~\ref{exceptional Chevalley} for finite index subgroups of certain split Chevalley groups, but this is work in progress.}
\item{Using Milnor's, Serre's and Bass' solution for the Congruence subgroup problem \cite[Theorem~3.6, Corollary~12.5]{MR244257} in the case of $R$ a ring of 
S-algebraic integers, the normal subgroup $N$ can be identified as the kernel of the reduction homomorphism $\pi_{2R}:{\rm Sp}_4(R)\to{\rm Sp}_4(R/2R)$
and hence $G/N={\rm Sp}_4(R/2R).$}
\end{enumerate}
\end{remark}

Let us finish the proof of Theorem~\ref{exceptional Chevalley} in case of ${\rm Sp}_4(R).$
First, note that $S\subset{\rm Sp}_4(R)$ being a normal generating set, implies both $\Pi(S)=\emptyset$ and $S$ mapping to a normal generating set in ${\rm Sp}_4(R)/N.$
Remember now, that ${\rm Sp}_4(R)$ is assumed to be boundedly generated by root elements. Hence to finish the proof of Theorem~\ref{exceptional Chevalley} for 
$\Phi=B_2$, we only have to prove that all root groups in ${\rm Sp}_4(R)$ are bounded with respect to $\|\cdot\|_S$ with a bound linear in $|S|.$ Let $\phi\in B_2$ be arbitrary. We know already by Proposition~\ref{mult_bound_B2} that the group $\{\varepsilon_{\phi}(2a)|a\in R\}$ is bounded (with respect to $\|\cdot\|_S$) with a bound
 linear in $\card{S}$. Furthermore, by Proposition~\ref{non-simply-laced-reduction}, we also know that for a set of representatives 
$X$ of $2R$ in $R$ the set $\{\varepsilon_{\phi}(x)|x\in X\}$ is bounded with a bound that is linear in $|S|$. 
Next, for each $a\in R$ there is an $x\in X$ and $b\in R$ such that $a=2b+x$ and hence the entire group $\varepsilon_{\phi}$ is bounded with a bound that is linear in $\card{S}.$ 

\subsection{The case of $G_2$}

First, we give the version of Proposition~\ref{mult_bound} for $G_2.$

\begin{Proposition}
\label{G2_mult_bound}
Let $R$ be a commutative ring with $1$ and let $S$ be a finite subset of $G_2(R)$ with $\Pi(S)=\emptyset$ and let $L(G_2)$ be $16$ times the constant $L(G_2)$ from Theorem~\ref{fundamental_reduction}. Then for all $a\in R:$
\begin{enumerate}
\item{$\|\varepsilon_{\phi}(2a)\|_S\leq L(G_2)\card{S}$ holds for all $\phi\in G_2$ short.}
\item{$\|\varepsilon_{\phi}(a)\|_S\leq L(G_2)\card{S}$ holds for all $\phi\in G_2$ long.}
\end{enumerate}
\end{Proposition}

\begin{proof}
Note that by Theorem~\ref{fundamental_reduction} there is a constant $L(G_2)$ such that for the ideal $I:=\sum_{A\in S}l(A)_3$, one has 
$I\subset\varepsilon(S,\chi,L(G_2)\card{S}).$ As before $\Pi(S)=\emptyset$ implies $I=R.$ This yields the claim of the proposition for long roots. To get the claim for short roots use part (1b) of Proposition~\ref{G2_ideals} and replace $L(G_2)$ by $16L(G_2).$ 
\end{proof}

Next, the analogue of Proposition~\ref{non-simply-laced-reduction}:

\begin{Proposition}
\label{non-simply-laced-reduction_G2}
Let $R$ be a commutative ring with $1$ such that $(R:2R)<\infty$ and let $G_2(R)$ be boundedly generated by root elements. 
Also let $S$ be a finite subset of $G_2(R)$ with $\Pi(S)=\emptyset$ and the property that $S$ maps to a normal generating subset of $G_2(R)/N$ for $N$ as in 
Lemma~\ref{congruence_fin} and let $F\subset R$ be finite. Then there is a constant $M(G_2,F,R)$ such that $\|\varepsilon_{\phi}(f)\|_S\leq M(G_2,F)\card{S}$ for all 
$f\in F$ and all $\phi\in B_2.$ So this holds in particular, if $F$ is a finite set of representatives of $2R$ in $R$. 
\end{Proposition}

The proof is essentially the same as the one of Proposition~\ref{non-simply-laced-reduction}, so we are going to omit it. Also completing the proof of 
Theorem~\ref{exceptional Chevalley} is very similar to ${\rm Sp}_4(R)$. The only difference is that we only have to show the boundedness of the root groups for the short roots, because it follows for long roots from Proposition~\ref{G2_mult_bound} already.

Thus, save for the proof of Theorem~\ref{fundamental_reduction}, we have proven Theorem~\ref{exceptional Chevalley}.
We also want to note the following corollary of the proof:

\begin{Corollary}
\label{sufficient_cond_gen_set}
Let $R$ be a commutative ring with $1$, $\Phi$ irreducible and of rank at least $2$ and assume $G(\Phi,R)=E(\Phi,R)$. Then a subset $S$ of $G$ normally generates $G$ precisely if 
\begin{enumerate}
\item{one has $\Pi(S)=\emptyset$ in case $\Phi\neq B_2,G_2$}
\item{one has $\Pi(S)=\emptyset$ and $S$ maps to a normally generating set of $G/N$ for $N$ as in Lemma~\ref{congruence_fin} in case $\Phi=B_2$ or $G_2.$}
\end{enumerate}
\end{Corollary}

\begin{remark}
\hfill
\begin{enumerate}
\item{
The case $\Phi\neq B_2, G_2$ is a consequence of a result by Abe \cite[Theorem~1,2,3,4]{MR991973}.}
\item{In case $\Phi=B_2$ or $G_2,$ the crucial point is that $\Pi(S)=\emptyset$ implies $N\subset\dl S\dr$. Hence it is obvious, that if $S$ maps to a normally generating set of $G/N$ for $G=G_2(R)$ or $Sp_4(R)$, then $S$ must normally generate $G.$ This is why we do not need the assumption $|R/2R|<+\infty$ here.}
\end{enumerate}
\end{remark}

The difference between ${\rm Sp}_4, G_2$ and the other cases is not merely an artifact of our proof strategy, as seen by studying the differences regarding normal generation between ${\rm Sp}_4, G_2$ and the other cases more in depth in Section~\ref{section_lower_bounds}. 

\section{Proof of Theorem~\ref{fundamental_reduction}}
\label{proof_fundamental_prop}

The main idea is that the claims of Theorem~\ref{fundamental_reduction} are first order statements and to use results about normal subgroups of split Chevalley groups and G\"odel's compactness theorem. We distinguish the three different cases of possible root systems $\Phi$ again.

\subsection{Level ideals for higher rank split Chevalley groups and ${\rm SL}_3$}\label{Centralization_higher}

This is the largest case. The main tool in this case is the following theorem by Abe.

\begin{Theorem}\cite[Theorem~1,2,3,4]{MR991973}
\label{Abe}
Let $\Phi$ be an irreducible root system that is not $A_1,B_2,G_2$ and let $R$ be a commutative ring with $1$. Then for each subgroup
$H\subset G(\Phi,R)$ normalized by the group $E(\Phi,R)$, there is an ideal $J\subset R$ and an additive subgroup of $L$ of $J$ 
such that $\bar{E}(J,L)\subset H\subset E^*(J,L).$ 
\end{Theorem}

\begin{remark}
\hfill
\begin{enumerate}
\item{
The paper \cite{MR843808} by Vaserstein deals with the simply laced case and with the multiple laced case under some assumptions. The papers Abe, Suzuki \cite{MR439947} and Abe \cite{MR258837} deal with local rings.}
\item{Theorem~\ref{Abe} is enough to prove strong boundedness of $G(\Phi,R)$ for commutative rings with $1$ and $\Phi\neq A_1,B_2,G_2$ with $G(\Phi,R)$ boundedly generated by root elements. However, this would not yield any linear bounds on $\Delta_k$ and is very similar to our argument, so we do not give more details.}
\end{enumerate}
\end{remark}

Next, we need the following lemma about root elements:

\begin{Lemma}
\label{A2_parts}
Let $\Phi$ an irreducible root system that is not $A_1,B_2$ or $G_2$, $R$ a commutative ring with $1$ and $A\in G(\Phi,R)$ be given and assume that $\lambda\in\varepsilon(A,\chi,N)$ for some $N\in\mathbb{N}$ and $\chi$ a short root. Then 
\begin{equation*}
\lambda R\subset\varepsilon(A,\chi,8N)
\end{equation*}
holds. 
\end{Lemma}

\begin{proof}
First note that $\lambda\in\varepsilon(A,,\chi,N)$ is equivalent to $\varepsilon_{\chi}(\lambda)\in B_A(N)$. We distinguish two cases:
\begin{enumerate}
\item
{$\Phi\neq B_n$ for $n\geq 3.$
The important fact is that $\chi$ is a short root in $\Phi$ and that all of these root systems contain a root subsystem isomorphic to $A_2$ consisting of short roots. Hence after conjugating with a suitable Weyl group elements, we can assume that $\Phi=A_2$ with simple positive roots $\alpha,\beta$ and 
$\chi=\alpha+\beta.$ But observe that $w_{\beta}(\alpha)=\chi$ and hence $\varepsilon_{\alpha}(\lambda)\in B_A(N).$ 
For $x\in R$ arbitrary, we obtain further 
\begin{equation*}
\varepsilon_{\chi}(\pm x\lambda)=\left(\varepsilon_{\alpha}(\pm\lambda),\varepsilon_{\beta}(\pm x)\right)\in B_A(2N).
\end{equation*}
}
\item{$\Phi=B_n$ for $n\geq 3.$ After conjugation with Weyl group elements, we assume that $n=3$ and so after conjugation we have positive, simple roots $\alpha,\beta,\chi$ with $\alpha,\beta$ long and $\chi$ short and $\beta$ double-bonded to $\chi$ in the Dynkin-diagram corresponding to the simple roots $\alpha,\beta$ and $\chi.$ However for $x\in R$ arbitrary 
\begin{equation}
\label{Bn_equation}
B_A(2N)\ni(\varepsilon_{\chi}(\lambda),\varepsilon_{\beta}(x))=\varepsilon_{\beta+\chi}(x\lambda)\varepsilon_{\beta+2\chi}(x\lambda^2).
\end{equation}
The root $\beta+\chi$ is short however and so conjugate to $\chi$ under the Weyl group action 
and hence we have $\varepsilon_{\beta+\chi}(\lambda)\in B_A(N).$ Thus for $x=1$ we obtain 
$\varepsilon_{\beta+2\chi}(\lambda^2)=\varepsilon_{\beta+\chi}(-\lambda)(\varepsilon_{\beta+\chi}(\lambda)\varepsilon_{\beta+2\chi}(\lambda^2))\in B_A(3N).$ 
The root $\beta+2\chi$ is long and hence $\varepsilon_{\beta+2\chi}(\lambda^2)$ is (up to sign) conjugate to $\varepsilon_{\beta}(\lambda^2)$ and so
$\varepsilon_{\beta}(\lambda^2)\in B_A(3N)$.
Yet $\alpha,\beta$ are simple roots in a root subsystem of $B_3$ isomorphic to $A_2$ and hence we obtain as in the first item that
$\varepsilon_{\beta}(x\lambda^2)\in B_A(6N)$ for all $x\in R.$ Summarizing this with equation (\ref{Bn_equation}) we get $\varepsilon_{\beta+\chi}(x\lambda)\in B_A(8N)$
for all $x\in R.$ Hence after conjugation we are done.
}
\end{enumerate}
\end{proof}

\begin{remark}
This Lemma is a more quantitative version of Vasersteins \cite[Theorem~4(a)]{MR843808}.
\end{remark}

Next, we want to prove the following technical proposition yielding the first part of Theorem~\ref{fundamental_reduction}:

\begin{Proposition}
\label{higher_rank_centralization}
Let $R$ be a commutative ring with $1,\Phi$ an irreducible root system that is not $A_1, B_2$ or $G_2$ and $\chi$ an arbitrary short root in $\Phi$. Then there is a constant $L(\Phi)\in~\mathbb{N}$ (not depending on $R$ or $A$ or $\chi$) such that for all $A\in G(\Phi,R)$ there is an ideal $I(A)$ with $I(A)\subset\varepsilon(A,\chi,L(\Phi))$ and $l(A)\subset\sqrt{I(A)}.$  
\end{Proposition}

\begin{proof}
First, choose polynomials $P$ in $\mathbb{Z}[y_{ij}]$ characterizing elements of $G(\Phi,\cdot)$ and $1\leq k,l\leq n_1+\cdots+n_u:=n$ with not both $k,l$ equal to $n$.

Next, let a language $\C L$ with the relation symbols, constants and function symbols 
\begin{equation*}
(\C R,0,1,+,\times,(a_{i,j})_{1\leq i,j\leq n},(e(k,l,v))_{v\in\mathbb{N}}) 
\end{equation*}
and a further function symbol $\cdot^{-1}:\C R^{n\times n}\to\C R^{n\times n}$ be given. Note that we use capital letters to denote matrices of variables (or constants) in the language in the following. For example the symbol $\C A$ denotes the $n\times n$-matrix of constants $(a_{i,j})$ and $X$ commonly refers to matrices of $n\times n$ variables in $\C L$. We also use the notation $X^{-1}:=^{-1}(X)$. 
Yet this is only a way to simplify notation, because first order sentences about matrices can always be reduced to first order sentences about their entries.

Let the first order theory $\C T_{kl}$ contain:
\begin{enumerate}
\item{Sentences forcing the universe $R:=\C R^{\C M}$ of each model $\C M$ of $\C T_{kl}$ is a commutative ring with respect to the functions 
$+^{\C M},\times^{\C M}$ and with $0^{\C M},1^{\C M}$ being $0$ and $1$.}
\item{For all $v\in\mathbb{N}$: 
If $k\neq l$ the sentence $e(k,l,v)=a_{k,l}^v$ should be included in $\C T_{kl}$. If on the other hand $k=l$, then choose the smallest $w\in\{1,\dots,u\}$ with 
$k<n_1+\cdots+n_w$ and include the sentence $e(k,l,v)=(a_{k,k}-a_{n_w,n_w})^v$.}
\item{The sentence $P(\C A)=0$.}
\item{The sentence $\forall X:(P(X)=0)\rightarrow (XX^{-1}=I_n),$ where $I_n$ denotes the unit matrix in $\C R^{n\times n}$ with entries the constant symbols $0,1$ as appropriate.}
\item{A family of sentences $(\theta_r)_{r\in\mathbb{N}}$ as follows: 
\begin{align*}
\theta_r:&\bigwedge_{1\leq v\leq r}\forall X_1^{(v)},\dots,X_r^{(v)},\forall e_1^{(v)},\dots,e_r^{(v)}\in\{0,1,-1\}:\\
&((P(X_1^{(v)})=\cdots=P(X_r^{(v)})=0)\rightarrow 
(\varepsilon_{\chi}(e(k,l,v))\neq (\C A^{e_1})^{X_1^{(v)}}\cdots(\C A^{e_r^{(v)}})^{X_r^{(v)}})
\end{align*}
Here $\C A^{1}:=\C A,\C A^{-1}:=\C A^{-1}$ and $\C A^{0}:=I_n.$
}
\end{enumerate} 

We first show that the theory $\C T_{kl}$ is inconsistent. To this end, let $\C M$ be a model for the sentences in (1) through (4)
and let $R:=R^{\C M}$ be the universe of $\C M.$ The sentences in (1) enforce that $R$ is a commutative ring with $1=1^{\C M}$ and $0=0^{\C M}$ and (3) enforces
that the matrix $A:=(a_{i,j}^{\C M})\in R^{n\times n}$ is an element of the split Chevalley group $G(\Phi,R).$ 
Let $H$ be the subgroup of $G(\Phi,R)$ normally generated by $A$. 
According to Theorem~\ref{Abe} there is a pair $(J,L)$ such that 
\begin{equation*}
\bar{E}(J,L)\subset H\subset E^*(J,L).
\end{equation*}
As $L\subset J$ holds, $A\in E^*(J,L)$ implies that $\pi_J(A)$ commutes with $E(R/J)$ and consequently that $\pi_{\sqrt{J}}(A)$ commutes with 
$E(R/\sqrt{J}).$ The ring $R/\sqrt{J}$ is reduced and so $\pi_{\sqrt{J}}(A)$ has the form described in Lemma~\ref{central_elements}.
This implies that $l(A)\subset\sqrt{J}.$ Hence as $\bar{E}(J,L)\subset H$, there is a constant $r'\in\mathbb{N}$ such that 
$\varepsilon_{\chi}(e(k,l,r')^{\C M})\in B_A(r').$ But this contradicts the statement $\theta_{r'}^{\C M}.$

So summarizing: a model of the sentences in (1) through (4) cannot be a model of all of the sentences $\theta_r$. Hence there is in fact no model of all of the above sentences and hence $\C T_{kl}$ is inconsistent. G\"odel's Compactness Theorem \cite{MR2596772} implies then, that a certain finite subset 
$\C T_{kl}^0\subset\C T_{kl}$ is already inconsistent. Hence there is only a finite collection of the $\theta_r$ contained in 
$\C T_{kl}^0.$ So let $L_{kl}(\Phi)\in\mathbb{N}$ be the largest $r\in\mathbb{N}$ with $\theta_r\in\C T_{kl}^0.$

For all $r\in\mathbb{N}$, we have $\{(1)-(4),\theta_{r+1}\}\vdash\theta_r.$ Hence the subset 
$\C T_{kl}^1\subset\C T_{kl}$ that contains all sentences in (1) through (4) and the \textit{single} sentence $\theta_{L_{kl}(\Phi)}$, must be inconsistent as well. 

Let $R$ be an arbitrary commutative ring with $1$ and let $A\in G(\Phi,R)$ be given. This gives us a model $\C M$ of (1) through (4) and hence as 
$\C T_{kl}^1$ is inconsistent, this model must violate the sentence $\theta_{L_{kl}(\Phi)}.$ Thus there are elements 
$g_1,\dots,g_{L_{kl}(\Phi)}\in G(\Phi,R)$ and $e_1,\dots,e_{L_{kl}(\Phi)}\in\{0,1,-1\}$ as well as a natural number $v\leq L_{kl}(\Phi)$
such that  
\begin{equation*}
\varepsilon_{\chi}(e(k,l,v)^{\C M})=(A^{e_1})^{g_1}\cdots (A^{e_{L_{kl}(\Phi)}})^{g_{L_{kl}(\Phi)}}.
\end{equation*} 
Hence we obtain that either a power of $a_{kl}$ (in case $k\neq l$) or a power of $a_{kk}-a_{n_1+\cdots+n_w,n_1+\cdots+n_w}$ (in case $k=l$) is an element of 
$\varepsilon(A,\chi,L_{kl}(\Phi)).$ So setting 
\begin{equation*}
L(\Phi):=\sum_{1\leq k,l\leq n\\ \text{ not both }k,l=n} 8L_{kl}(\Phi),  
\end{equation*}
we get together with Lemma~\ref{A2_parts} an ideal $I(A)$ in $R$ such that $I(A)\subset\varepsilon(A,\chi,L(\Phi))$ and $l(A)\subset\sqrt{I(A)}$ holds. This ideal $I(A)$  has the desired properties as stated in the proposition for the single root $\chi.$ Note, that all short roots are conjugate under elements of the Weyl group and hence we have 
$\varepsilon(A,\chi_1,L(\Phi))=\varepsilon(A,\chi_2,L(\Phi))$ for two short roots $\chi_1,\chi_2$ in $\Phi$, so the conclusion does not depend on the specific shoort
$\chi.$
\end{proof}

\begin{remark}
Compare this result with Morris' result \cite[Theorem~6.1(1)]{MR2357719}. In case of $\Phi=A_n$ and $R$ an order in a ring of algebraic integers, 
Proposition~\ref{higher_rank_centralization} is a consequence of \cite[Theorem~6.1(1)]{MR2357719} and \cite[Theorem~6.4]{MR2357719} by way of considering the normal subgroup $N:=\dl A\dr.$
\end{remark}

\subsection{Level ideals for ${\rm Sp}_4$}

Remember that $B_2$ has the positive roots $\alpha,\beta,\alpha+\beta$ and $\chi=2\alpha+\beta$ with $\alpha$ short and 
$\beta$ long and both simple. Again, we invoke a compactness argument. The main ingredient is the following observation due to Costa and Keller instead of Theorem~\ref{Abe}:

\begin{Theorem}\cite[Theorem~2.6, 4.2, 5.1, 5.2]{MR1162432}
\label{Keller_B2}
Let $R$ be a commutative ring with $1$. Let $A\in{\rm Sp}_4(R)$ be given. Then for all $x\in l(A)$ one has 
$\varepsilon_{\chi}(2x+x^2)\varepsilon_{\alpha+\beta}(x^2)\subset\dl A\dr_{E(B_2,R)}$, where $\dl A\dr_{E(B_2,R)}$ denotes the subgroup of ${\rm Sp}_4(R)$ generated by the
$E(B_2,R)$-conjugates of $A.$ 
\end{Theorem}

Root elements in ${\rm Sp}_4$ are more complicated than in higher rank groups:

\begin{Lemma}
\label{B2_ideals}
Let $R$ be a commutative ring with $1$ and $S\subset{\rm Sp}_4(R).$ Let $\lambda\in R$ and 
$N\in\mathbb{N}$ be given. Then 
\begin{enumerate}
\item{
$\varepsilon_{\chi}(2\lambda+\lambda^2)\varepsilon_{\alpha+\beta}(\lambda^2)\in B_S(N)$ implies 
$\{\varepsilon_{\chi}(2x\lambda^2)|x\in R\}\subset B_S(2N).$}
\item{$\varepsilon_{\chi}(\lambda)\in B_S(N)$ implies $\varepsilon_{\phi}(\lambda)\in B_S(3N)$ for all $\phi$ short.}
\item{$\varepsilon_{\alpha}(x\lambda)\in B_S(N)$ for all $x\in R$ implies $\varepsilon_{\phi}(x\lambda^2)\in B_S(3N)$ for all $\phi$ long and all $x\in R$.}
\item{$\varepsilon_{\chi}(\lambda)\in B_S(N)$ implies $\{\varepsilon_{\chi}(2x\lambda)|x\in R\}\subset B_S(6N).$} 
\item{$\varepsilon_{\chi}(2\lambda+\lambda^2)\varepsilon_{\alpha+\beta}(\lambda^2)\in B_S(N)$ implies 
$\{\varepsilon_{\phi}(2x\lambda^2)|x\in R,\phi\in B_2\}\subset B_S(6N)$.}
\end{enumerate}
All of the above implications stay true, if the balls $B_S$ are replaced by a normal subgroup of ${\rm Sp}_4(R).$
\end{Lemma}

\begin{proof}
For the first part inspect the commutator $(\varepsilon_{\alpha}(x),\varepsilon_{\chi}(2\lambda+\lambda^2)\varepsilon_{\alpha+\beta}(\lambda^2))$
for $x\in R$ arbitrary. 
For the second part, note that $\varepsilon_{\chi}(\lambda)$ is conjugate to $\varepsilon_{\beta}(\lambda)$ and so
$\varepsilon_{\beta}(\lambda)\in~B_S(N).$ Note further
\begin{equation*}
B_A(2N)\ni(\varepsilon_{\beta}(\lambda),\varepsilon_{\alpha}(1))=\varepsilon_{\alpha+\beta}(\pm\lambda)\varepsilon_{\chi}(\pm\lambda).
\end{equation*}
These two facts imply $\varepsilon_{\alpha+\beta}(\lambda)\in B_S(3N).$ The element $\varepsilon_{\alpha+\beta}(\lambda)$ is conjugate to $\varepsilon_{\phi}(\lambda)$
for every short root $\phi\in B_2$. This proves the second part and the third part follows by considering for $x\in R$ the commutator
\begin{equation*}
B_A(2N)\ni(\varepsilon_{\beta}(x),\varepsilon_{\alpha}(\lambda))=\varepsilon_{\alpha+\beta}(\pm x\lambda)\varepsilon_{\chi}(\pm x\lambda^2).
\end{equation*}
and noting $\varepsilon_{\alpha+\beta}(\pm x\lambda)\in B_A(N).$
For the fourth part, note that we have by the second part, that $\varepsilon_{\alpha}(\lambda)\in B_S(3N).$ Next inspect for $x\in R$ the commutator:
\begin{equation*}
B_S(6N)\ni(\varepsilon_{\alpha}(\lambda),\varepsilon_{\alpha+\beta}(x))=\varepsilon_{\chi}(2x\lambda).
\end{equation*}
This proves the fourth part. The last part follows from part (1) and (2).
\end{proof}

With this lemma, the second case of Theorem~\ref{fundamental_reduction} follows:

\begin{Proposition}
\label{B2_centralization}
Let $R$ be a commutative ring with $1$ and let $A\in{\rm Sp}_4(R)$ be given. Then there is a constant $L(B_2)$ (not depending on $A$ or $R$) such that
$2l(A)_2\subset\varepsilon(A,\phi,L(B_2))$ for $\phi\in B_2$ arbitrary. 
\end{Proposition}

\begin{proof}
The proof is very similar to the one of Proposition~\ref{higher_rank_centralization}.  
First, let natural numbers $k,l$ be given with $1\leq k,l\leq 4$. Also if $k=l,$ then we assume that $k=l<4.$
The language $\C L$ and the theory $\C T_{kl}$ is defined the same way as in Proposition~\ref{higher_rank_centralization} except for three differences:
First we include a constant symbol $e(k,l)$ instead of $e(k,l,v)$. Secondly, (2) has the form 
\begin{equation*}
e(k,l)=\left\{\begin{array}{lr}
        a_{kl}, & \text{if } k\neq l \\
        a_{kk}-a_{k+1,k+1}, & \text{if not}
        \end{array}\right.
\end{equation*}

Most importantly, (5) is a family of sentences $(\theta_r)_{r\in\mathbb{N}}$ such that 
\begin{align*}
\theta_r:&\forall X_1,\dots,X_r,\forall e_1,\dots e_r\in\{0,1,-1\}:((P(X_1)\wedge\dots\wedge P(X_r))\rightarrow\\ 
&(\varepsilon_{\chi}(2e(k,l)+e(k,l)^2)\varepsilon_{\alpha+\beta}(e(k,l)^2)\neq
(\C A^{e_1})^{X_1}\cdots(\C A)^{e_r})^{X_r}))
\end{align*}

Invoking Theorem~\ref{Keller_B2} instead of Theorem~\ref{Abe} yields that a model of (1) through (4) cannot be a model of all sentences in (5).
Hence $\C T_{kl}$ is inconsistent. Using G\"odel's compactness, we obtain, as in the proof of Proposition~\ref{higher_rank_centralization}, that
 there is an $L_{k,l}(B_2)$ such that the subset $\C T_{kl}^1\subset\C T_{kl}$ that contains all sentences in (1) through (4) and the \textit{single} sentence 
$\theta_{L_{k,l}(B_2)}$ is already inconsistent. 

Let $R$ be an arbitrary commutative ring with $1$ and let $A\in {\rm Sp}_4(R)$ be given. This gives us a model $\C M$ of the (1) through (4) and hence as $\C T_1$ is inconsistent this model must violate the statement $\theta_{L_{k,l}(B_2)}^{\C M}.$ Thus there are elements 
$g_1,\dots,g_{L_{k,l}(B_2)}\in{\rm Sp}_4(R)$ and $e_1,\dots e_{L_{k,l}(B_2)}\in\{0,1,-1\}$ such that (abusing the notation slightly)  
\begin{equation*}
\varepsilon_{\chi}(2e(k,l)+e(k,l)^2)\varepsilon_{\alpha+\beta}(e(k,l)^2)=(A^{e_1})^{g_1}\cdots (A^{e_{L_{k,l}(B_2)}})^{g_{L_{k,l}(\Phi)}}
\end{equation*} 
Next, Lemma~\ref{B2_ideals}(5) implies $2(e(k,l)^2)\in \varepsilon(A,\phi,6L_{k,l}(B_2))$ for all $\phi\in B_2.$
If we sum over all admissible $k,l$, this implies for all $\phi\in B_2$ that
\begin{equation*}
2l(A)_2=\sum_{k,l}(2e(k,l)^2)\subset\varepsilon(A,\phi,\sum_{k,l}6L_{k,l}(B_2)).
\end{equation*} 
So defining $L(B_2):=\sum_{k,l}6L_{k,l}(B_2)$, we get the statement.
\end{proof}

\subsection{Level ideals for $G_2$}

Remember that the positive roots in $G_2$ are
$\alpha,\beta,\alpha+\beta,2\alpha+\beta,3\alpha+\beta$ and $3\alpha+2\beta=\chi$ for $\alpha,\beta$ simple, positive roots in $G_2$ with $\alpha$ short, $\beta$ long. Also note that the roots $3\alpha+\beta$ and $\beta$ span a root subsystem of $G_2$ isomorphic to $A_2.$ 

Next, we are using the following result by Costa and Keller: 

\begin{Theorem}\cite[(3.6)~Main Theorem]{MR1487611}
\label{Keller_G2}
Let $R$ be a commutative ring with $1$ and let $H$ be an $E(G_2,R)$-normalized subgroup of $G_2(R).$ Then there is a pair of ideals $J,J'$ in $R$ with 
\begin{equation*}
(x^3,3x|x\in J)\subset J'\subset J
\end{equation*}
such that 
\begin{equation*}
[E(R),E(J,J')]\subset H\subset G(J,J').
\end{equation*}

We are not defining $G(J,J')$, but note that $H\subset G(J,J')$ implies that $H$ becomes trivial after reducing mod $J$. 
\end{Theorem}

This implies:

\begin{Corollary}
\label{simplified_Keller_G2}
Let $R$ be a commutative ring with $1$, $A\in G_2(R)$ and $H$ the smallest subgroup of $G_2(R)$ normalized by $E(G_2,R)$ and containing $A$. Then we have 
$\varepsilon_{3\alpha+2\beta}(a^3),\varepsilon_{3\alpha+2\beta}(3a)\in H$ for all $a\in l(A).$
\end{Corollary}

\begin{proof}
This follows directly from Theorem~\ref{Keller_G2}. Note first that $J$ must contain $l(A)$, because $A\in H$ becomes scalar after reducing modulo $J.$ Hence for $a\in l(A)$ we get $3a,a^3\in J'$ for all $a\in l(A).$  
Lastly, $\{\varepsilon_{\beta}(b)|\ b\in J'\}\subset H$ holds, because $\beta$ is a root in the long $A_2$ in $G_2.$ 
\end{proof}

Next, note the following:

\begin{Proposition}
\label{G2_ideals}
Let $R$ be a commutative ring with $1$ and let $S\subset G_2(R)$ be given. Then 
\begin{enumerate}
\item{
if for $N\in\mathbb{N},\lambda\in R$ one has $\varepsilon_{\chi}(\lambda)\in B_S(N)$, then
\begin{enumerate}
\item{$\{\varepsilon_{\phi}(x\lambda)|x\in R\}\subset B_S(2N)$ for $\phi$ long and}
\item{$\{\varepsilon_{\phi}(2x\lambda)|x\in R\}\subset B_S(16N)$ for $\phi$ short holds.}
\end{enumerate}}
\item{if $\varepsilon_{\alpha}(\lambda)\in B_S(N)$, then $\{\varepsilon_{\chi}(x\lambda^3),|x\in R\}\subset B_S(4N)$.}
\end{enumerate}
The implications are still true, if the balls $B_S$ are replaced by a normal subgroup of $G_2(R).$
\end{Proposition}

\begin{proof}
Part (1a) can be obtained by arguing as in the $A_2$-case. For part (1b) inspect the following commutator formula for all $x\in R:$
\begin{equation*}
\varepsilon_{\alpha+\beta}(\pm x\lambda)\varepsilon_{2\alpha+\beta}(\pm x^2\lambda)\varepsilon_{3\alpha+\beta}(\pm x^3\lambda)
\varepsilon_{3\alpha+2\beta}(\pm x^3\lambda^2)
=(\varepsilon_{\beta}(\lambda),\varepsilon_{\alpha}(x))\in B_S(2N)
\end{equation*}
Note that $\varepsilon_{3\alpha+\beta}(x^3\lambda),\varepsilon_{3\alpha+2\beta}(x^3\lambda^2)\in B_S(2N)$ by part (1a) and hence
\begin{equation*}
\varepsilon_{\alpha+\beta}(x\lambda)\varepsilon_{2\alpha+\beta}(x^2\lambda)\in B_S(2N*3)=B_S(6N).
\end{equation*}
Taking the commutator of this product with $\varepsilon_{\alpha}(1)$ yields (up to conjugation)
\begin{equation*}
\varepsilon_{2\alpha+\beta}(2x\lambda)\varepsilon_{3\alpha+\beta}(3x^2\lambda+6x\lambda)\varepsilon_{3\alpha+2\beta}(3x^2\lambda^2)\in B_S(12N).
\end{equation*}
Yet we have again $\varepsilon_{3\alpha+\beta}(3x^2\lambda+6x\lambda),\varepsilon_{3\alpha+2\beta}(3x^2\lambda^2)\in B_S(2N)$ and hence
$\varepsilon_{2\alpha+\beta}(2x\lambda)\in B_S(16N).$
Lastly, for part (2) inspect first the commutator 
\begin{equation*}
B_S(2N)\ni(\varepsilon_{\alpha}(\lambda),\varepsilon_{\beta}(x))=
\varepsilon_{\alpha+\beta}(\pm x\lambda)\varepsilon_{2\alpha+\beta}(\pm x\lambda^2)\varepsilon_{3\alpha+\beta}(\pm x\lambda^3)
\varepsilon_{3\alpha+2\beta}(\pm x^2\lambda^2).
\end{equation*}
However, all of the factors besides $\varepsilon_{3\alpha+\beta}(x\lambda^3)$ in this product commute with $\varepsilon_{\beta}(1).$
Thus taking the commutator with $\varepsilon_{\beta}(1)$, we obtain the claim after conjugation.
\end{proof}

With this in hand, the last part of Theorem~\ref{fundamental_reduction} follows:

\begin{Proposition}
\label{G2_centralization}
Let $R$ be a commutative ring with $1$ and let $A\in G_2(R)$ be given. Then there is a constant $L(G_2)$ (not depending on $A$ or $R$) such that
$l(A)_3\subset\varepsilon(A,\chi,L(G_2)).$
\end{Proposition}

\begin{proof}
The proof is very similar to the ones in the previous subsections. Let natural numbers $k,l$ be given with 
$1\leq k,l\leq 8$. Also if $k=l,$ we further assume that $k=l<8.$

Aside from the places, where $k,l$ have to range between $1$ and $8$, the language $\C L$ and the theory $\C T_{kl}$ is defined the same way as in Proposition~\ref{B2_centralization} except (5) has the form:

A family of sentences $(\theta_r)_{r\in\mathbb{N}}$ such that 
\begin{equation*}
\theta_r:\forall X_1,\dots,X_r,\forall e_1,\dots e_r\in\{0,1,-1\}:((P(X_1)\wedge\dots\wedge P(X_r))\rightarrow 
(\varepsilon_{\beta}(e(k,l)^3)\neq
(\C A^{e_1})^{X_1}\cdots(\C A)^{e_r})^{X_r}))
\end{equation*}

One obtains invoking Corollary~\ref{simplified_Keller_G2} that a model of (1) through (4) cannot be a model of (5).
Hence $\C T_{kl}$ is inconsistent. As before, we can by invoking compactness find an $L_{k,l}(G_2)\in\mathbb{N}$ such that the subset 
$\C T_{kl}^{1}\subset\C T_{kl}$ that contains all sentences in (1) through (4) and the \textit{single} sentence $\theta_{L_{k,l}(G_2)}$, is already inconsistent. 

Next, let $R$ be an arbitrary commutative ring with $1$ and let $A\in G(\Phi,R)$ be given. This gives us a model $\C M$ of the sentences in (1) through (4) and hence as $\C T_{kl}^1$ is inconsistent this model must violate the statement $\theta_{L_{k,l}(G_2)}^{\C M}.$ Thus there are elements 
$g_1,\dots,g_{L_{k,l}(G_2)}\in G_2(R)$ and $e_1,\dots e_{L_{k,l}(G_2)}\in\{0,1,-1\}$ such that (abusing the notation slightly)  
\begin{equation*}
\varepsilon_{\beta}(e(k,l)^3)=(A^{e_1})^{g_1}\cdots (A^{e_{L_{k,l}(G_2)}})^{g_{L_{k,l}(G_2)}}
\end{equation*} 
Proposition~\ref{G2_ideals}(1a) implies $(e(k,l)^3)\subset\varepsilon(A,\chi,2L_{k,l}(G_2).$
Summing further over all admissible $k,l$ implies 
\begin{equation*}
\sum_{k,l}(e(k,l)^3)\subset\varepsilon(A,\chi,\sum_{k,l}2L_{k,l}(G_2)).
\end{equation*} 
Define next $L(G_2):=\sum_{k,l}2L_{k,l}(G_2)$ and we are done.
\end{proof}

\begin{remark}
In this paper, we restrict ourselves to the simply-connected type of split Chevalley groups. However, a careful rereading of the proofs reveals, that the we only used the fact that we have a description of $G(\Phi,R)$ as a matrix group and consequently an explicit description of the level ideal. However, similar descriptions exist for a lot of other types of split Chevalley groups and consequently 
statements similar to Theorem~\ref{exceptional Chevalley} can be obtained for them. 
\end{remark}

\section{Applications, corollaries and variants of Theorem~\ref{exceptional Chevalley}}
\label{Corollaries}

Theorem~\ref{exceptional Chevalley} naturally raises the question of which rings $R$ fulfill its main assumption, 
that is bounded generation by root groups for $G(\Phi,R)$. We will mostly deal with rings of stable range $1$ and rings of S-algebraic integers for this.

\subsection{Stable range $1$, semilocal rings and uniform boundedness}

A useful tool in this context is the following (slightly reformulated) observation due to Tavgen:

\begin{Proposition}\cite[Proposition~1]{MR1044049}
\label{bootstrapping}
Let $\Phi$ be a root system, $R$ a commutative ring with $1$ such that there is an $m:=m(R),N(R)\in\mathbb{N}$ with the property that 
each irreducible root subsystem $\Phi_0$ of $\Phi$ generated by simple roots of $\Phi$ with rank $m$ satisfies 
\begin{equation*}
\|E(\Phi_0,R)\|_{EL}\leq N(R){\rm rank}(\Phi_0).
\end{equation*}
Then $\|E(\Phi,R)\|_{EL}\leq N(R){\rm rank}(\Phi).$ Here $\|\cdot\|_{EL}$ denotes the word norm on $E(\Phi,R)$ with respect to the generating set given by root elements.
\end{Proposition}

This is a useful proposition, because it allows to obtain bounded generation results for higher ranks from such results for low-rank root systems 
at least if $G(\Phi,R)=E(\Phi,R)$ holds. In particular, we have the following:

\begin{Corollary}
\label{rank_1_boundedness}
Let $R$ be a commutative ring with $1$ such that ${\rm SL}_2(R)=G(A_1,R)$ is boundedly generated by root elements. Then for all irreducible root systems $\Phi$ the elementary Chevalley group $E(\Phi,R)$ is boundedly generated by root elements.
\end{Corollary}

\begin{proof}
This follows from Proposition~\ref{bootstrapping} in the case $m(R)=1$ for $N(R)$ determined by the bounded generation of ${\rm SL}_2(R),$
because for each simple root $\alpha\in\Phi$ the subgroup $E(\{\alpha,-\alpha\},R)$ of $E(\Phi,R)$ is either isomorphic to ${\rm SL}_2(R)$ or a quotient of it. 
\end{proof}
 
Next, we define stable range:

\begin{mydef}
The \textit{stable range} of a commutative ring $R$ with $1$ is the smallest $n\in\mathbb{N}$ with the following property:
If any $v_0,\dots,v_m\in R$ generate the unit ideal $R$ for $m\geq n$, then there are $t_1,\dots,t_m$ such that the elements 
$v_1':=v_1+t_1v_0,\dots,v_m':=v_m+t_mv_0$ also generate the unit ideal. If the ring $R$ does not have stable range $1$, but for each $a\in R-\{0\}$
the ring $R/aR$ does, then $R$ is said to have stable range $3/2.$
If no such $n$ exists, $R$ has stable range $+\infty.$
\end{mydef}

\begin{remark}
Having stable range at most $m$ for $m\in\mathbb{N}$ or at most $3/2$ are first order properties.
\end{remark}

Note the following result:

\begin{Proposition}\cite[Lemma~9]{MR961333}
\label{Vaserstein_decomposition}
Let $R$ be a commutative ring with $1$ of stable range $1$ and $\Phi$ a root system. Then ${\rm SL}_2(R)$ is boundedly generated by root elements and
$\|{\rm SL}_2(R)\|_{EL}\leq 4.$
\end{Proposition}

This proposition together with Corollary~\ref{rank_1_boundedness} yields that $E(\Phi,R)$ is boundedly generated by root elements for all irreducible root 
systems $\Phi$ and using Tavgen's original version of Proposition~\ref{bootstrapping} that each element in $E(\Phi,R)$ can be written as a product of at most four upper and lower unitriangular elements. This was first observed by Vavilov, Smolenski, Sury in \cite[Theorem~1]{MR2822515}.  

\begin{Proposition}\label{root_generation_semilocal_rings}
\cite[Corollary~2.4]{MR439947}
Let $R$ be a semi-local ring. Then for all irreducible root systems $\Phi$ of rank greater than one, the group $G(\Phi,R)$ is generated by root elements.  
\end{Proposition}

Also each semilocal ring has stable range $1$:

\begin{Lemma}\cite[Lemma~6.4, Corollary~6.5]{MR0174604}
Every semilocal ring, that is each ring with only finitely many maximal ideals has stable range $1.$ So also each field has stable range $1$.
\end{Lemma}

So for $R$ a semilocal ring the group $G(\Phi,R)$ is boundedly generated by root elements and hence Theorem~\ref{exceptional Chevalley} can be applied to $G(\Phi,R)$. This case is more structured than Theorem~\ref{exceptional Chevalley} in fact:

\begin{Theorem}
\label{semilocal_uniform}
Let $R$ be a commutative, semilocal ring with $1$ and let $\Phi$ an irreducible root system. Furthermore, assume if $\Phi=B_2$ or $G_2$ that 
$(R:2R)<\infty$ holds. Then $G(\Phi,R)$ is uniformly bounded.
\end{Theorem}

\begin{proof}
The strategy is to find a constant $K\in\mathbb{N}$ such that each finite normally generating subset $S$ of $G:=G(\Phi,R)$ has a subset
$\bar{S}$ with $|\bar{S}|\leq K$ such that $\bar{S}$ is also a normally generating subset of $G(\Phi,R).$ Then 
Theorem~\ref{exceptional Chevalley} yields:
\begin{equation*}
\|G(\Phi,R)\|_S\leq\|G(\Phi,R)\|_{\bar{S}}\leq C(\Phi,R)|\bar{S}|\leq C(\Phi,R)K. 
\end{equation*}
and so uniform boundedness for $G(\Phi,R)$ holds.

Assume $R$ has precisely $m$ maximal ideals. Let $S$ be a finite set of normal generators of $G(\Phi,R).$ Lemma~\ref{necessary_cond_conj_gen} implies 
$\Pi(S)=\emptyset.$ Observe that for all $T_1,T_2\subset G(\Phi,R)$, we have $\Pi(T_1\cup T_2)=\Pi(T_1)\cap\Pi(T_2).$
This implies that if there are only $m$ maximal ideals in $R,$ then already some subset $S'$ of $S$
with $\card{S'}\leq m+1$ has the property $\bigcap_{A\in S'}\Pi(A)=\emptyset.$ Hence in case $\Phi\neq B_2$ or $G_2$, Corollary~\ref{sufficient_cond_gen_set}(1)
tells us that $S'$ is already a normally generating subset of $G(\Phi,R)$. This finishes the case $\Phi\neq B_2, G_2.$ 

Next, we do the case $\Phi=B_2$ or $G_2.$ We have $(R:2R)<\infty$ by assumption and hence Lemma~\ref{congruence_fin} implies for 
$N:=\dl\varepsilon_{\phi}(2a)|a\in R,\phi\in \Phi\dr$, that the group $G/N$ is finite. The set $S$ normally generates the group $G$ and hence the image of $S$ in 
$G/N$ normally generates $G/N$ and so we can pick a subset $S''\subset S$ with at most $M:=\card{G/N}$ elements such that the image of $S''$ in $G/N$ normally generates $G/N.$ Hence considering the set $\bar{S}:=S'\cup S''$ we have 
\begin{equation*}
|\bar{S}|\leq|S'|+|S''|\leq m+1+M
\end{equation*}
and the upper bound $m+1+M$ clearly does not depend on $S.$ Corollary~\ref{sufficient_cond_gen_set}(2) implies that $\bar{S}$ is a normally generating set of $G(\Phi,R)$. Thus we are done.
\end{proof}

\begin{remark}
The above theorem applies to various cases for example local rings, $p$-adic integers or other discrete valuation domains.
\end{remark}

We also obtain the following:

\begin{Theorem}
\label{stable_range1_strong_boundedness}
Let $R$ be a commutative ring with $1$ of stable range $1$ and $\Phi$ an irreducible root system of rank at least $2$ that is not $G_2$ or $B_2.$
Then for the elementary subgroup $E(\Phi,R)$ of $G(\Phi,R)$, there is a constant $C(\Phi,R)$ such that
\begin{equation*}
\Delta_k(E(\Phi,R))\leq C(\Phi,R)k
\end{equation*}
for all $k\in\mathbb{N}$. If $\Phi=B_2$ or $G_2$ we must further assume that $(R:2R)<\infty.$
\end{Theorem}

\begin{proof}
We want to show a version of Theorem~\ref{exceptional Chevalley} that speaks about $E(\Phi,R)$ instead of $G(\Phi,R).$ This can be done by following the same arguments. The only difference in the proofs of Proposition~\ref{higher_rank_centralization}, Proposition~\ref{B2_centralization} and Proposition~\ref{G2_centralization} is that in the sentences of the theories $\C T_{kl}$, one can no longer quantify over the full Chevalley group $G(\Phi,R)$, but must quantify over all elements of $E(\Phi,R)$ despite the fact that this group cannot be defined in first order terms. Stable range $1$ is a first order property of a ring however and we know that $E(\Phi,R)$ has bounded generation by root elements for such rings. Thus by including a collection of sentences that describes that 
$R$ has stable range $1$, we can modify the $\C T_{kl}$ such that it quantifies over all elements of $E(\Phi,R)$ by quantifying over appropriate finite products of root elements and then the rest of the argument goes through. 
\end{proof}

\subsection{Rings of S-algebraic integers}

First, we are going to define S-algebraic integers.

\begin{mydef}\cite[Chapter~I, §11]{MR1697859}\label{S-algebraic_numbers_def}
Let $K$ be a finite field extension of $\mathbb{Q}$. Then let $S$ be a finite subset of the set $V$ of all valuations of $K$ such that $S$ contains all archimedean valuations. Then the ring $\C O_S$ is defined as 
\begin{equation*}
\C O_S:=\{a\in K|\ \forall v\in V-S: v(a)\geq 0\}.
\end{equation*}
\end{mydef}

Rings of S-algebraic integers do not have stable range $1$. For a remarkably large class of them -the ones with infinitely many units- the corresponding ${\rm SL}_2$ are still boundedly generated by root elements \cite[Theorem~1.1]{MR3892969}. This will be more relevant in our upcoming paper 
\cite{Hessenbergform_explicit_strong_bound}, when we talk about explicit bounds. More important for us is the following classical result:

\begin{Theorem}\cite[Theorem~A]{MR1044049}
\label{Tavgen}
Let $\Phi$ be an irreducible root system of rank at least $2$ and $R$ a ring of S-algebraic integers in a number field. Then $G(\Phi,R)$ has bounded generation with respect to root elements.
\end{Theorem}

Furthermore, all non-zero ideals $I$ in a ring $R$ of S-algebraic integers have finite index. So, rings $R$ of S-algebraic integers in number fields have the property that $G(\Phi,R)$ is boundedly generated by root elements for all irreducible $\Phi$ of rank at least $2$ and the ideal $2R$ (and all other non-zero ideals) have finite index in $R.$ Hence Theorem~\ref{exceptional Chevalley} can be applied to the groups $G(\Phi,R)$. This gives us the following Theorem:

\begin{Theorem}
\label{alg_numbers_strong_bound}
Let $R$ be a ring of S-algebraic integers in a number field and $\Phi$ an irreducible root system of rank at least $2$. 
Then there is a constant $C(\Phi,R)\geq 1$ such that for $G(\Phi,R)$ one has
\begin{equation*}
\Delta_k(G(\Phi,R))\leq C(\Phi,R)k
\end{equation*} 
for all $k\in\mathbb{N}.$ 
\end{Theorem}

\begin{remark}
Note that in contrast to the result by K\k{e}dra, Libman and Martin \cite[Theorem~6.1]{KLM} 
we do not have any control on the the behaviour of $C(\Phi,R)$, beyond the fact that it does not depend on $k.$ However, in the paper \cite{Hessenbergform_explicit_strong_bound} we will remedy this fact by providing explicit values for $C(\Phi,R)$ in case the ring is a principal ideal domain.
\end{remark}

We provide lower bounds on $\Delta_k(G(\Phi,R))$ in Section~\ref{section_lower_bounds}. 

Next, we are going to talk about orders in rings of algebraic integers and Morris results in \cite{MR2357719} and how to use them to get strong boundedness results. We do not define orders precisely, but they are subrings of rings of algebraic integers that are also sublattices of the same ring of algebraic integers.
First, there are the following results by Morris that are very similar to our results.

\begin{Theorem}\cite[Theorem~6.1, Remark~6.2, Corollary~6.13]{MR2357719}
Let $B$ be an order in a ring of algebraic integers and $S$ a multiplicative set in $B-\{0\}$. Further assume either that $n\geq 3$
or that $S^{-1}B$ has infinitely many units. Also let $X$ be a subset of $G:={\rm SL}_n(S^{-1}B)$, that is normalized by root elements and that does not consist entirely of scalar matrices. Then $X$ boundedly generates a finite index subgroup $N$ of ${\rm SL}_n(S^{-1}B)$ with a bound on the maximal length of a word in elements of $X$ that depends on $n,$ the degree $[K:\mathbb{Q}],$ the minimal numbers of generators of the level ideal $l(N)$ and the cardinality of $S^{-1}B/l(N).$ If 
$X:=\{gsg^{-1}|s\in S,g\in {\rm SL}(S^{-1}B)\}$ for a finite set $S$ the minimal number of generators of $l(N)$ is smaller than $n^2\card{S}.$

Similarly if $\Gamma$ is a finite index subgroup of ${\rm SL}_n(S^{-1}B)$ and $X\subset\Gamma$ a set that is normalized by $\Gamma$ and does not consist entirely of scalar matrices, then $X$ boundedly generates a finite index subgroup $N$ of $\Gamma$ with a bound that depends on the same numbers as above.
\end{Theorem}

\begin{remark}
In the terminology of our paper this establishes that finite index, finitely normally generated, non-central subgroups $N$ of ${\rm SL}_n(S^{-1}B)$ are strongly bounded
(or $\Delta_k(N)<\infty$ for all $k$). The main difference (in the case that a finite $S$ normally generates ${\rm SL}_n$) to our result, is that Morris has no control on the actual value of $\Delta_k({\rm SL}_n)$, whereas we can establish that the dependence is at least linear in $k$. 
Structurally, the main reason for this difference seems to be that Morris uses a first order compactness result upon the entirety of a \textit{set of generators} to establish bounded generation but does not have any control on any particular one of the generators individually. We, on the other hand, apply a compactness result 
upon a particular given element of the group $G(\Phi,R)$ to obtain root elements with arguments lying in its level ideal and only later consider the full generating set to obtain the missing elements in case of $B_2$ and $G_2$. Our methods are able to prove a stronger version of \cite[Theorem~6.1, Remark~6.2, Corollary~6.13]{MR2357719} as well, but this is work in progress.
\end{remark}

Morris \cite[Theorem~5.26]{MR2357719} proves bounded generation by root elements for the subgroup $E(A_1,R)$ of ${\rm SL}_2(R)$, even in the case that $R$ is only a localization of an order, if said localization has infinitely many units. He demonstrates further that the elementary subgroup $E(A_2,R)$ of ${\rm SL}_3(R)$ is boundedly generated by root elements for $R$ a localization of an order 
\cite[Corollary~3.13]{MR2357719}. Most importantly however, in both cases the bounded generation results follows by proving that localization of orders satisfy certain
first order properties, that Morris calls ${\rm Gen}(t,r)$ and ${\rm Exp}(t,l)$ in case of $E(A_2,R)$ and additionally ${\rm Unit}(1,x)$ in case of $E(A_1,R)$ as well as in both cases stable range $3/2.$ Then Morris show that these properties imply bounded generation by root elements. 

Hence adding these first order properties in the proofs of Proposition~\ref{higher_rank_centralization}, Proposition~\ref{B2_centralization} and Proposition~\ref{G2_centralization} and applying Corollary~\ref{rank_1_boundedness} in case $\Phi$ is not simply-laced, one can prove the following:

\begin{Proposition}
Let $R$ be a localization of an order in a ring of algebraic integers and $\Phi$ an irreducible root system of rank at least $2$. 
Assume further that $R$ has infinitely many units in case $\Phi$ is not simply-laced. There is a constant $C(\Phi,R)$ such that 
\begin{equation*}
\Delta_k(E(\Phi,R))\leq C(\Phi,R)k
\end{equation*}
holds for all $k\in\mathbb{N}$.
\end{Proposition}

Lastly, it should be possible to prove bounded generation results and hence strong boundedness also in the case of rings of functions of
algebraic curves over finite fields. The picture seems to be less clear in this area however as the only result about this we could find was \cite{Nica}
stating bounded generation of ${\rm SL}_n(\mathbb{F}[T])$ for $\mathbb{F}$ a finite field and $n\geq 3.$

\section{Lower bounds on $\Delta_k$}
\label{section_lower_bounds}

In this section, we talk about lower bounds on $\Delta_k$. The dichotomy between $G_2,B_2$ and the other $\Phi$ persists here. Namely, for $\Phi=B_2$ or $G_2$ the lower bounds depend strongly on the ring $R.$ First the higher rank cases:

\begin{Proposition}\cite[Theorem~6.1]{KLM}
\label{Lower_bounds_higher_rank}
Let $\Phi$ be an irreducible root system of rank at least $2$ and let $R$ be a Dedekind domain with finite class number and infinitely many maximal ideals such that 
$G(\Phi,R)$ is boundedly generated by root elements. Further assume that $2$ is a unit in $R$ if $\Phi=B_2$ or $G_2.$
Then $\Delta_k(G(\Phi,R))\geq k$ for all $k\in\mathbb{N}.$ 
\end{Proposition}

\begin{proof}
Let $k$ distinct maximal ideals $\C P_1,\dots,\C P_k$ be given and let $c$ be the class number of $R.$ All the ideals $\C P_i^c$ are principal so choose $t_i$
as one of its generators and set for all $i$
\begin{equation*}
r_i:=\prod_{1\leq j\neq i\leq k}t_j.
\end{equation*}
Fix a short root $\phi\in\Phi$ and consider the elements $A_i:=\varepsilon_{\phi}(r_i)$ and the set $S:=\{A_1,\dots,A_k\}.$ Then 
$\Pi(A_i)=\bigcup_{j\neq i}\C \{P_j\}$ holds and thus $\Pi(S)=\emptyset.$ Hence if $\Phi\neq B_2, G_2$, then Corollary~\ref{sufficient_cond_gen_set}(1) implies that $S$ is a normally generating set of $G(\Phi,R).$ If on the other hand $\Phi=B_2$ or $G_2$ holds, then the assumption on $2$ implies 
$R=2R$ and so the condition in Corollary~\ref{sufficient_cond_gen_set}(2) reduces to $\Pi(S)=\emptyset$ as well.

To finish the proof, assume for contradiction that $\|\varepsilon_{\phi}(1)\|_S\leq k-1.$ Then there are elements $g_1,\dots,g_{k-1}\in G(\Phi,R)$ 
and $s_1,\dots,s_{k-1}\in S\cup S^{-1}\cup\{1\}$ such that
\begin{equation*}
\varepsilon_{\phi}(1)=\prod_{1\leq i\leq k-1}s_i^{g_i}.
\end{equation*} 
However, $\Pi(s_i^{g_i})=\Pi(s_i)$ contains at least $k-1$ elements of $\{\C P_1,\dots,\C P_k\}$ and hence $\bigcap_{1\leq i\leq k-1}\Pi(s_i^{g_i})$ cannot possibly 
be empty. This implies 
\begin{equation*}
\emptyset\neq\bigcap_{1\leq i\leq k-1}\Pi(s_i^{g_i})\subset\Pi(\varepsilon_{\phi}(1))=\emptyset.
\end{equation*}
This contradiction yields $\|\varepsilon_{\phi}(1)\|_S\geq k$. This proves as $\card{S}=k$ that 
\begin{equation*}
\Delta_k(G(\Phi,R))\geq\|G(\Phi,R)\|_S\geq~\|\varepsilon_{\phi}(1)\|_S\geq k.
\end{equation*}
\end{proof}

\begin{remark}
This is a generalization of \cite[Theorem~6.1]{KLM}, yet the proof is essentially the same.
\end{remark}

Next, we are going to describe lower bounds on $\Delta_k({\rm Sp}_4(R))$ and $\Delta_k(G_2(R))$ in the general case. It turns out that in this case the (existence of) lower bounds are strongly dependent on the way $2$ splits into primes in the ring $R.$  

\begin{Theorem}
\label{lower_bounds_rank2}
Let $\Phi$ be $B_2$ or $G_2$ and let $R$ be a ring of S-algebraic integers in a number field with $R\neq 2R$. 
Further let
\begin{equation*}
r:=r(R):=|\{\C P|\ \C P\text{ divides 2R, is a prime ideal and }R/\C P=\mathbb{F}_2\}|
\end{equation*}
be given. Then for $G(\Phi,R)$
\begin{enumerate}
\item{the inequality $\Delta_k(G(\Phi,R))\geq k$ holds for all $k\in\mathbb{N}$ with $k\geq r(R)$ and}
\item{the equality $\Delta_k(G(\Phi,R))=-\infty$ holds for $k<r(R).$} 
\end{enumerate}
\end{Theorem}

We show both parts of the theorem separately. For the first part, the main difficulty, compared to Proposition~\ref{Lower_bounds_higher_rank} 
comes, from the more complex conditions a set $S$ has to fulfill to be a normal generating set. To address this, we need the following technical Proposition describing algebraic properties of finite quotients of rings of S-algebraic integers.

\begin{Proposition}
\label{quotient_normal_generation}
Let $R$ be a ring of S-algebraic integers and $\C P_1,\dots,\C P_s$ be non-zero prime ideals and $l_1,\dots,l_s\in\mathbb{N}.$
Assume further that at most one of the $\C P_i$ has the property 
\begin{equation*}
[R/\C P_i:\mathbb{F}_2]=1
\end{equation*}
 and let $\bar{x}\in R/(\C P_1^{l_1}\cdots\C P_s^{l_s})=:\bar{R}$ 
be a unit. Then $\varepsilon_{\alpha}(\bar{x})$ normally generates ${\rm Sp}_4(\bar{R})$ or $G_2(\bar{R})$ respectively.
\end{Proposition}

\begin{proof}
First, we do the case ${\rm Sp}_4(R).$ Let $N$ be the subgroup of ${\rm Sp}_4(R)$ normally generated by $\varepsilon_{\alpha}(\bar{x}).$
We first prove for $\bar{R}_0:=\{\bar{y}\in\bar{R}|\varepsilon_{\alpha}(\bar{y})\in N\}$ that $\bar{R}_0=\bar{R}.$ This is done in two steps.
First, we prove that $\bar{R}_0$ contains all units of $\bar{R}$ and is closed under addition and then second, that $\bar{R}$ is generated as an additive group by its
units. But this yields the proposition, because $\varepsilon_{\alpha}(\bar{a})\in N$ for all $\bar{a}\in\bar{R}$ implies together with Lemma~\ref{B2_ideals}(3),
that $N$ contains all root elements and by Proposition~\ref{root_generation_semilocal_rings} the group ${\rm Sp}_4(\bar{R})$ is generated by its root elements as 
$\bar{R}$ is finite and hence semi-local.

For the first step, according to \cite[Lemma~20(c), Chapter~3, p.~23]{MR3616493}, we have for any unit $\bar{u}\in\bar{R}$ that
\begin{equation*}
N\ni h_{\beta}(\bar{x}\bar{u}^{-1})\varepsilon_{\alpha}(\bar{x})h_{\beta}(\bar{x}\bar{u}^{-1})^{-1}
=\varepsilon_{\alpha}((\bar{x}\bar{u}^{-1})^{\langle\alpha,\beta\rangle}\bar{x})
=\varepsilon_{\alpha}((\bar{x}\bar{u}^{-1})^{-1}\bar{x})=\varepsilon_{\alpha}(\bar{u}).
\end{equation*}

For the second step, observe first that $\bar{R}$ does not have $\mathbb{F}_2\times\mathbb{F}_2$ as a quotient ring. This is the case, because otherwise
the ring $R$ would have two distinct non-zero prime ideals $\C Q_1,\C Q_2$ with $R/\C Q_1=R/\C Q_2=\mathbb{F}_2$ and $\C P_1^{l_1}\cdots\C P_s^{l_s}\subset\C Q_1,\C Q_2.$ 
But then $\C Q_1$ and $\C Q_2$ are among the $\C P_1,\dots,\C P_s$, which is impossible, because there is at most one $\C P_i$ with $R/\C P_i=\mathbb{F}_2$.
Yet semi-local rings without $\mathbb{F}_2\times\mathbb{F}_2$ as a quotient ring are generated by their units according to \cite[Lemma~2(d)]{MR2161255} 
and hence $\bar{R}_0=\bar{R}.$

For the case $G_2(R)$, note that we obtain $\{\bar{x}|\varepsilon_{\alpha}(\bar{x})\in N\}=\bar{R}$ the same way as in the case of ${\rm Sp}_4(R).$
So $N$ contains all root elements for short roots. Lemma~\ref{G2_ideals}(2) yields now that $N$ also contains all of the root elements for long roots.   
Hence as $G_2(\bar{R})$ is generated by root elements we are done.
\end{proof}

We can show the first part of the theorem now.

\begin{proof}
Let the ideal $2R$ in $R$ split into primes as follows:
\begin{equation*}
2R=\left(\prod_{i=1}^r\C P_i^{l_i}\right)\cdot\left(\prod_{j=1}^s\C Q_j^{k_j}\right)
\end{equation*}
with $[R/\C P_i:\mathbb{F}_2]=1$ for $1\leq i\leq r$ and $[R/\C Q_j:\mathbb{F}_2]>1$ for $1\leq j\leq s.$ Next, let $c$ be the class number of $R.$
Pick elements $x_1,\dots,x_r\in R$ such that $\C P_i^c=(x_i)$ for all $i.$
Also choose $r+1$ distinct primes $V_{r+1},\dots,V_k$ in $R$ which do not agree with any of the $\C P_1,\dots,\C P_r,\C Q_1,\dots,\C Q_s.$
Passing to the powers $V_{r+1}^c,\dots,V_k^c$ we can find elements $v_{r+1},\dots,v_k\in R$ with $V_{r+1}^c=(v_{r+1}),\dots,V_k^c=(v_k).$ 
Further, define the following elements for $1\leq u\leq r(R)=r$ 
\begin{equation*}
r_u:=\left(\prod_{1\leq i\neq u\leq r}x_i\right)\cdot v_{r+1}\cdots v_k.
\end{equation*}
For $k\geq u\geq r+1$ set
\begin{equation*}
r_u:=x_1\cdots x_r\cdot\left(\prod_{r+1\leq u\neq q\leq k} v_q\right).
\end{equation*}

We consider the set $S:=\{\varepsilon_{\alpha}(r_1),\dots,\varepsilon_{\alpha}(r_k)\}$ in ${\rm Sp}_4(R)$ or $G_2(R).$ Note that $\alpha$ is the short, positive simple root in both cases. For the sake of brevity, we will only write down the case of ${\rm Sp}_4(R).$
\begin{Claim}S is a normal generating set of ${\rm Sp}_4(R).$
\end{Claim}

According to Corollary~\ref{sufficient_cond_gen_set}(2), we have to fulfill two conditions for this claim to hold, first $\Pi(S)=\emptyset$ and second that 
$S$ maps to a normally generating subset of ${\rm Sp}_4(R)/N$ for $N:=\dl\varepsilon_{\phi}(2x)|x\in R,\phi\in B_2\dr$.

First, note that 
\begin{align*}
\Pi(\varepsilon_{\alpha}(r_u))= 
\begin{cases} 
\{\C P_1,\dots,\hat{\C P_u},\dots,\C P_r,V_{r+1},\dots,V_k\} &\text{ , if }1\leq u\leq r\\
\{\C P_1,\dots,\C P_r, V_{r+1},\dots,\hat{V_u},\dots,V_k\} & \text{ , if } r+1\leq u\leq k, 
\end{cases}
\end{align*}
where the hat denotes the omission of the corresponding prime. This implies $\Pi(S)=\emptyset.$

For the second condition note that Milnor's, Serre's and Bass' solution for the Congruence subgroup problem \cite[Theorem~3.6, Corollary~12.5]{MR244257} implies that 
\begin{equation*}
N=ker(\pi_{2R}:{\rm Sp}_4(R)\to {\rm Sp}_4(R/2R)). 
\end{equation*}
Hence it suffices to show that under the reduction homomorphism $\pi_{2R}:{\rm Sp}_4(R)\to{\rm Sp}_4(R/2R)$ the set $S$ maps to a normally generating set of 
${\rm Sp}_4(R/2R).$ Using the Chinese Remainder Theorem yields:
\begin{equation*}
{\rm Sp}_4(R/2R)=\prod_{i=1}^r {\rm Sp}_4\left(R/(\C P_i^{l_i})\right)\times{\rm Sp}_4\left(R/(\prod_{j=1}^s\C Q_j^{k_j})\right).
\end{equation*}
For $1\leq u\leq k$, it follows further: 
\begin{align*}
\pi_{2R}(\varepsilon_{\alpha}(r_u))=\left(\bigtimes_{i=1}^r(\varepsilon_{\alpha}(r_u+\C P_i^{l_i})),\varepsilon_{\alpha}(r_u+\prod_{j=1}^s\C Q_j^{k_j})\right)
\end{align*}
Depending on $u$ these elements look quite different. 

First for $u=1$ the element $r_1$ is divisible by all $\C P_i^{l_i}$ except for $i=1$. Hence this implies
\begin{equation*}
\pi_{2R}(\varepsilon_{\alpha}(r_1))=
\left(\varepsilon_{\alpha}(r_1+\C P_1^{l_1}),\bigtimes_{i=2}^r(\varepsilon_{\alpha}(0)),\varepsilon_{\alpha}(r_1+\prod_{j=1}^s\C Q_j^{k_j})\right)
=\left(\varepsilon_{\alpha}(r_1+\C P_1^{l_1}),1,\varepsilon_{\alpha}(r_1+\prod_{j=1}^s\C Q_j^{k_j})\right)
\end{equation*}
Phrased differently, it is the element $\varepsilon_{\alpha}(r_1+\C P_1^{l_1}\prod_{j=1}^s\C Q_j^{k_j})$ in the subgroup 
${\rm Sp}_4(R/(\C P_1^{l_1}\prod_{j=1}^s\C Q_j^{k_j})).$
Note that $r_1$ is not divisible by any of the $\C Q_j$ nor $\C P_1$ and hence $r_1+\C P_1^{l_1}\prod_{j=1}^s\C Q_j^{k_j}$ is a unit in 
$R/(\C P_1^{l_1}\prod_{j=1}^s\C Q_j^{k_j}).$ Thus by Proposition~\ref{quotient_normal_generation} the element 
$\varepsilon_{\alpha}(r_1+2R)$ normally generates the subgroup 
\begin{equation*}
{\rm Sp}_4\left(R/(\C P_1^{l_1}\prod_{j=1}^s\C Q_j^{k_j})\right)={\rm Sp}_4(R/\C P_1^{l_1})\times {\rm Sp}_4\left(R/(\prod_{j=1}^s\C Q_j^{k_j})\right)
\end{equation*}
of ${\rm Sp}_4(R/2R).$ 

The same way for $2\leq u\leq r$ it follows that the element $\varepsilon_{\alpha}(r_u+2R)$ normally generates the subgroup 
${\rm Sp}_4(R/\C P_u^{l_u})\times{\rm Sp}_4(R/(\prod_{j=1}^s\C Q_j^{k_j}))$ of ${\rm Sp}_4(R/2R).$
So already the subset $\{\varepsilon_{\alpha}(r_1)),\dots,\varepsilon_{\alpha}(r_r)\}$ of $S$ maps to a normally generating subset of ${\rm Sp}_4(R/2R)$ under $\pi_{2R}.$ This proves the claim.

\begin{Claim}The diameter of $\|\cdot\|_S$ is at least $k.$ As $|S|=k$ this proves the first part of the theorem for ${\rm Sp}_4(R).$
\end{Claim}

Assume for contradiction that ${\rm diam}(\|\cdot\|_S)\leq k-1.$ Then there are elements $s_1,\dots,s_{k-1}\in S\cup S^{-1}\cup\{1\}$ and 
$g_1,\dots,g_{k-1}\in{\rm Sp}_4(R)$ with
\begin{equation*}
\varepsilon_{\alpha}(1)=\prod_{1\leq i\leq k-1}s_i^{g_i}.
\end{equation*} 
But the set $\Pi(s_i^{g_i})=\Pi(s_i)$ contains at least $k-1$ elements of the set $\{\C P_1,\dots,\C P_r, V_{r+1},\dots,V_k\}$ and hence 
$\bigcap_{1\leq i\leq k-1}\Pi(s_i^{g_i})$ is not empty and so $\emptyset\neq\bigcap_{1\leq i\leq k-1}\Pi(s_i^{g_i})\subset\Pi(\varepsilon_{\alpha}(1))=\emptyset.$
This contradiction proves $\|\varepsilon_{\alpha}(1)\|_S\geq k$.
\end{proof}

For the second part of Theorem~\ref{lower_bounds_rank2}, note the following: 

\begin{Lemma}
There is an epimorphism ${\rm Sp}_4(\mathbb{F}_2)\to \mathbb{F}_2$ with $\varepsilon_{\phi}(a)\mapsto a$ for all $a\in\mathbb{F}_2$ and $\phi\in B_2.$
Similarly there is an epimorphism $G_2(\mathbb{F}_2)\to\mathbb{F}_2$ with 
\begin{align*}
\varepsilon_{\phi}(a)\mapsto 
\begin{cases} 
a &\text{ ,if }\phi\in G_2\text{ short}\\
0 &\text{ ,if }\phi\in G_2\text{ long}
\end{cases}
\end{align*}
\end{Lemma}

\begin{proof}
We only do the case ${\rm Sp}_4(\mathbb{F}_2)$ again. According to \cite[Theorem~8;Chapter~6,p.~43]{MR3616493}, the group ${\rm Sp}_4(\mathbb{F}_2)$ is isomorphic to the group $G$ generated by elements of order $2$ named $\varepsilon_{\alpha}(1)$,$\varepsilon_{\beta}(1)$,
$\varepsilon_{\alpha+\beta}(1)$, $\varepsilon_{2\alpha+\beta}(1)$, $\varepsilon_{-\alpha}(1)$,$\varepsilon_{-\beta}(1)$ and $\varepsilon_{-\alpha-\beta}(1),\varepsilon_{-2\alpha-\beta}(1)$ subject to relations of the form
\begin{align*}
&(\varepsilon_{\phi}(1),\varepsilon_{\psi}(1))=1\text{, if }\phi+\psi\in B_2\text{ and no other sum of positive multiples of }\psi\text{ and }\phi\text{ is a root}\\
&(\varepsilon_{\phi}(1),\varepsilon_{\psi}(1))=1\text{, if }\phi+\psi\notin B_2\text{ and }\phi+\psi\neq 0\\
&(\varepsilon_{\phi}(1),\varepsilon_{\psi}(1))=\varepsilon_{\phi+\psi}(1)\varepsilon_{\tau}(1)\text{, if }\phi+\psi\in B_2
\text{ and }\tau=\phi+2\psi\text{ or }2\phi+\psi\in B_2.
\end{align*}

The map $\{\varepsilon_{\phi}(a)\mapsto a\}$ sends both sides of these relations to the same element, namely $0$, 
and hence we get an epimorphism as required.
\end{proof}

\begin{remark}
\hfill
\begin{enumerate}
\item{The group ${\rm Sp}_4(\mathbb{F}_2)$ is isomorphic to the permutation group $S_6$ and the epimorphism in the lemma is the 
sign homomorphism $S_6\to\mathbb{F}_2.$}
\item{The group $G_2(\mathbb{F}_2)$ has a simple subgroup $U$ with $[G_2:U]=2$. The homomorphism is the map $G_2(\mathbb{F}_2)\to G_2(\mathbb{F}_2)/U=\mathbb{F}_2.$
The group $U$ is isomorphic to the twisted group $^2A_2(\mathbb{F}_9).$
} 
\end{enumerate}
\end{remark}

Using this, the second part of the theorem follows:

\begin{proof}
We restrict ourselves to the case ${\rm Sp}_4(R)$ again. Let $2R=(\prod_{i=1}^r\C P_i^{l_i})(\prod_{j=1}^s\C Q_j^{k_j})$ be given as in the proof of the first part of the theorem. Using the Chinese Remainder Theorem, we know that the map 
\begin{equation*}
{\rm Sp}_4(R)\twoheadrightarrow{\rm Sp}_4(R/2R)=\prod_{i=1}^r{\rm Sp}_4(R/(\C P_i^{l_i}))\times\prod_{i=1}^r{\rm Sp}_4(R/(\C Q_j^{k_j}))
\twoheadrightarrow\prod_{i=1}^r {\rm Sp}_4(R/\C P_i)={\rm Sp}_4(\mathbb{F}_2)^r
\end{equation*}
is an epimorphism. So composing with the epimorphism ${\rm Sp}_4(\mathbb{F}_2)\to\mathbb{F}_2$, we obtain an epimorphism $g:{\rm Sp}_4(R)\to\mathbb{F}_2^r.$ 
This suffices to prove the second part of the theorem, because a given normally generating set $S$ of ${\rm Sp}_4(R)$ with $|S|\leq r-1$ would map to a generating set of $\mathbb{F}_2^r$ with less than $r$ elements. The group $\mathbb{F}_2^r$ cannot be generated by less than $r$ elements however. 
\end{proof}

This finishes the proof of Theorem~\ref{lower_bounds_rank2}. We note the following corollary:

\begin{Corollary}
Let $R$ be a ring of S-algebraic integers and $r(R)$ defined as in Theorem~\ref{lower_bounds_rank2}. Then both ${\rm Sp}_4(R)$ and $G_2(R)$ have 
abelianization $\mathbb{F}_2^r.$
\end{Corollary}

\begin{proof}
We only do the case ${\rm Sp}_4(R).$ Note that $\dl\varepsilon_{\phi}(2x)|x\in R,\phi\in B_2\dr\subset ({\rm Sp}_4(R),{\rm Sp}_4(R))$ 
by Lemma~\ref{B2_ideals}(4) and (2) and further that ${\rm Sp}_4(R)$ is boundedly generated by root elements by Theorem~\ref{Tavgen}. Thus the abelianization $A(R)$ of ${\rm Sp}_4(R)$ 
is a finitely generated, $2$-torsion group. Let $r':=dim_{\mathbb{F}_2}(A(R)).$ The proof of Theorem~\ref{lower_bounds_rank2} implies
that $A(R)$ has the quotient $\mathbb{F}_2^r$ and hence $r'\geq r.$ Now on the other hand $r'>r$ is impossible, because it would imply as in the proof of the second part of Theorem~\ref{lower_bounds_rank2} that there are no normal generating sets of ${\rm Sp}_4(R)$ with precisely $r$ elements, which is wrong.  
\end{proof}

For rings of quadratic integers it is known how $2$ splits into primes and hence we can give the following complete description of $r(R)$:

\begin{Corollary}
\label{other_quadratic_roots}
Let $D$ be a square-free integer and $R$ the ring of algebraic integers in $\mathbb{Q}[\sqrt{D}].$ Then the value of $r(R)$ is 
\begin{enumerate}
\item{$r(R)=1$ precisely if $D\equiv 2,3,5,6,7\text{ mod }8$, so $\Delta_1({\rm Sp}_4(R)),\Delta_1(G_2(R))\neq -\infty.$}
\item{$r(R)=2$ precisely if $D\equiv 1\text{ mod }8$, so $\Delta_1({\rm Sp}_4(R))=\Delta_1(G_2(R))=-\infty$ and $\Delta_2({\rm Sp}_4(R))=\Delta_2(G_2(R))>-\infty.$}
\end{enumerate}
\end{Corollary}

\begin{proof}
We obtain from \cite[Theorem~25]{MR3822326} that the ideal $2R$ splits and ramifies in $R$ as follows:
\begin{enumerate}
\item{$2R$ is inert precisely if $D\equiv 5\text{ mod }8.$}
\item{$2R$ ramifies precisely if $D\equiv 2,3,6,7\text{ mod }8.$}
\item{$2R$ splits precisely if $D\equiv 1\text{ mod }8.$}
\end{enumerate}
In the first two cases, this implies $r(R)=1$ and in the third case $r(R)=2.$
\end{proof}

We finish this section with the following explicit example:

\begin{Corollary}\label{small_generating_sets_Sp4}
Let $R=\mathbb{Z}[\frac{1+\sqrt{-7}}{2}]$ be the ring of algebraic integers in the number field $\mathbb{Q}[\sqrt{-7}].$ Then ${\rm Sp}_4(R)$ and $G_2(R)$ are not generated by a single conjugacy class and so $\Delta_{1}({\rm Sp}_4(R))=\Delta_{1}(G_2(R))=-\infty$.
\end{Corollary}

\section*{Closing remarks}

This paper gives rise to a question regarding generalizations of the stated results. It is natural to ask how the results generalize to rings $R$ of integers of global fields of positive characteristic instead of rings of algebraic integers. This poses two issues:
\begin{enumerate}
\item{First, bounded generation of the corresponding Chevalley group $G(\Phi,R).$ There is the result of Nica \cite{Nica} regarding the basic case of $R=\mathbb{F}[T]$ for $\mathbb{F}$ finite, but no general result in this direction is known to us.
However replacing the number theoretic results used in the proofs of \cite{MR2357719} or \cite{MR1044049} by corresponding results in the 'number theory of R'
this should be possible. It seems likely to us that there might be some issues in case of the critical characteristics $2$ and $3$ for some root systems 
$\Phi.$}
\item{Secondly, our argument in case $\Phi=B_2, G_2$ relied on $R/2R$ being finite. This is clearly true in the case of $char(R)\geq 3$ as 
$R/2R$ is trivial in this case. If $char(R)=2,$ then this fails.}
\end{enumerate}

Next, one can ask about more general S-arithmetic lattices than the ones we dealt with. Here S-arithmetic lattices means groups commensurable with $G(\Phi,R).$ There is one straightforward way of generalizing our results to finite index subgroups by way of analyzing the subnormal structure of the Chevalley groups instead of the normal structure as we did, but this strategy will break down quickly after only a couple of easy examples. 
The main issue is that the subnormal group structures are- to the best of our knowledge- badly understood for arbitrary commutative rings in the moment. There seem to be some results in this direction notably \cite{Hong} and \cite{Vavilov_Subnormal}.
Also it might be possible to imitate the strategy of Morris and try to isolate certain first order properties of rings of algebraic integers that might facilitate such a strategy.

Beyond this, there is also the issue that boundedness are badly behaved under passage to finite index supergroups and subgroups. Further it is not clear to us what algebraic structure plays the role of the level ideals of the corresponding subgroup in question in this case. 

Considering the fact, that we are mainly interested in lattices it seems likely that a more geometric interpretation of our results is the most straightforward path to a generalization. 

\section*{Appendix}

\begin{proof}[Lemma~\ref{central_elements}]
We split the proof into three parts. First we are going to show the statement for fields, then for integral domains and finally for general reduced rings.
So let $K$ be a field and $A=(a_{kl})\in G(\Phi,K)$ be given. For fields one has $G(\Phi,K)=E(\Phi,K)$ by \cite[Corollary~2.4]{MR439947} and hence $A$ is central in 
$G(\Phi,K).$ Then by \cite[Lemma~28, Chapter~3,p.~29]{MR3616493} there are $t_1,\dots,t_u\in K-\{0\}$ such that $A=\prod_{i=1}^u h_{\alpha_i}(t_i)$, where 
$\{\alpha_1,\dots,\alpha_u\}=\Pi$ are the simple, positive roots in $\Phi.$ Further
\begin{equation}\label{central_elements_eq1}
1=\prod_{i=1}^u t_i^{\langle\phi,\alpha_i\rangle}\text{ for all }\phi\in\Phi.
\end{equation}
Furthermore, for $1\leq j\leq u$ the element $A$ acts on the component of $K^{n_j}\subset K^{n_1+n_2+\dots+n_u}$ 
associated to the highest weight $\lambda_j$ of $V_j$ by multiplication with
\begin{equation}\label{central_elements_eq2}
\prod_{i=1}^u t_i^{\langle\lambda_j,\alpha_i\rangle}=\prod_{i=1}^u t_i^{\delta_{ij}}=t_j.
\end{equation}
Here we use that $\lambda_j$ is chosen as the fundamental weight corresponding to $\alpha_j$, that is $\langle\lambda_j,\alpha_i\rangle=\delta_{ij}$ holds for all $1\leq i,j\leq u.$
Each other weight in $V_j$ has the form $\lambda_j-\sum\phi$, where the $\phi$ are positive roots in $\Phi.$ Then (\ref{central_elements_eq1})
and (\ref{central_elements_eq2}) imply that $A$ acts on $K^{n_j}$ by $t_j I_{n_j}.$ So this yields the claim for fields.

For integral domains $R$, we distinguish two cases: 
\begin{case} R is finite.
\end{case}
Yet finite integral domains are fields and hence we are done.
\begin{case} R is infinite.
\end{case}
Let $\alpha\in\Phi$ be given and observe that for $K$ the algebraic closure of the quotient field of $R$ we have the map 
\begin{equation*}
\phi_{\alpha}:\mathbb{G}_a(K)=K\to G(\Phi,R),\lambda\mapsto (A,\varepsilon_{\alpha}(\lambda)).
\end{equation*}
This is a morphism of algebraic varieties and note that as $A$ commutes with elements in $\varepsilon_{\alpha}(R)$ by assumption, 
we have $\phi_{\alpha}|_R$ is equal to the identity. But $R$ is Zariski-dense in $\mathbb{G}_a(K).$
So $\phi_{\alpha}|_R$ being the identity implies that $\phi_{\alpha}$ is constant. Hence $A$ commutes with the entire group $\varepsilon_{\alpha}$ in 
$G(\Phi,K).$ However $G(\Phi,K)$ is generated by the elements $\{\varepsilon_{\alpha}(\lambda)|\lambda\in K,\alpha\in\Phi\}$. Hence 
$A$ is central in $G(\Phi,K)$, so we are done again.

Lastly, let $R$ be a reduced ring. Further let $\C P$ be a prime ideal in $R.$ So $\pi_{\C P}(A)\in G(\Phi,R/\C P)$ commutes with $E(\Phi,R/\C P)$ and
$R/\C P$ is an integral domain. Thus we obtain 
\begin{equation*}
A\equiv(a_{11}I_{n_1})\oplus\cdots\oplus(a_{n_1+\cdots+n_{u-1}+1,n_1+\cdots+n_{u-1}+1}I_{n_u})\text{ mod}\C P
\end{equation*}
for all prime ideals $\C P.$ This implies 
\begin{equation*}
A\equiv(a_{11}I_{n_1})\oplus\cdots\oplus(a_{n_1+\cdots+n_{u-1}+1,n_1+\cdots+n_{u-1}+1}I_{n_u})\text{ mod}\bigcap_{\C P\text{ prime in }R}\C P=\sqrt{(0)}.
\end{equation*}
However $R$ is reduced and so $\sqrt{(0)}=(0)$ holds and we are done.
\end{proof}

\bibliography{bibliography}

\def\polhk#1{\setbox0=\hbox{#1}{\ooalign{\hidewidth
  \lower1.5ex\hbox{`}\hidewidth\crcr\unhbox0}}}
  \def\polhk#1{\setbox0=\hbox{#1}{\ooalign{\hidewidth
  \lower1.5ex\hbox{`}\hidewidth\crcr\unhbox0}}}
  \def\polhk#1{\setbox0=\hbox{#1}{\ooalign{\hidewidth
  \lower1.5ex\hbox{`}\hidewidth\crcr\unhbox0}}}
  \def\polhk#1{\setbox0=\hbox{#1}{\ooalign{\hidewidth
  \lower1.5ex\hbox{`}\hidewidth\crcr\unhbox0}}}
  \def\polhk#1{\setbox0=\hbox{#1}{\ooalign{\hidewidth
  \lower1.5ex\hbox{`}\hidewidth\crcr\unhbox0}}}
  \def\polhk#1{\setbox0=\hbox{#1}{\ooalign{\hidewidth
  \lower1.5ex\hbox{`}\hidewidth\crcr\unhbox0}}} \def\cprime{$'$}
\begin{thebibliography}{10}

\bibitem{MR258837}
Eiichi Abe.
\newblock Chevalley groups over local rings.
\newblock {\em Tohoku Math. J. (2)}, 21:474--494, 1969.

\bibitem{MR991973}
Eiichi Abe.
\newblock Normal subgroups of {C}hevalley groups over commutative rings.
\newblock In {\em Algebraic {$K$}-theory and algebraic number theory
  ({H}onolulu, {HI}, 1987)}, volume~83 of {\em Contemp. Math.}, pages 1--17.
  Amer. Math. Soc., Providence, RI, 1989.

\bibitem{MR439947}
Eiichi Abe and Kazuo Suzuki.
\newblock On normal subgroups of {C}hevalley groups over commutative rings.
\newblock {\em Tohoku Math. J. (2)}, 28(2):185--198, 1976.

\bibitem{MR0174604}
H.~Bass.
\newblock {$K$}-theory and stable algebra.
\newblock {\em Inst. Hautes \'{E}tudes Sci. Publ. Math.}, (22):5--60, 1964.

\bibitem{MR244257}
H.~Bass, J.~Milnor, and J.-P. Serre.
\newblock Solution of the congruence subgroup problem for {${\rm
  SL}_{n}\,(n\geq 3)$} and {${\rm Sp}_{2n}\,(n\geq 2)$}.
\newblock {\em Inst. Hautes \'{E}tudes Sci. Publ. Math.}, (33):59--137, 1967.

\bibitem{MR1611814}
Claude Chevalley.
\newblock Certains sch\'{e}mas de groupes semi-simples.
\newblock In {\em S\'{e}minaire {B}ourbaki, {V}ol. 6}, pages Exp. No. 219,
  219--234. Soc. Math. France, Paris, 1995.

\bibitem{MR1162432}
Douglas~L. Costa and Gordon~E. Keller.
\newblock Radix redux: normal subgroups of symplectic groups.
\newblock {\em J. Reine Angew. Math.}, 427:51--105, 1992.

\bibitem{MR1487611}
Douglas~L. Costa and Gordon~E. Keller.
\newblock On the normal subgroups of {$G_2(A)$}.
\newblock {\em Trans. Amer. Math. Soc.}, 351(12):5051--5088, 1999.

\bibitem{MR961333}
R.~K. Dennis and L.~N. Vaserstein.
\newblock On a question of {M}. {N}ewman on the number of commutators.
\newblock {\em J. Algebra}, 118(1):150--161, 1988.

\bibitem{Gal-Kedra-Trost}
\'Swiatos{\l}aw~R. Gal, Jarek K\k{e}dra, and Alexander Trost.
\newblock Finite index subgroups in {C}hevalley groups are bounded: an addendum
  to "on bi-invariant word metrics".
\newblock https://arxiv.org/abs/1808.06376.

\bibitem{MR0396773}
James~E. Humphreys.
\newblock {\em Linear algebraic groups, corrected fifth printing}.
\newblock Springer-Verlag, New York-Heidelberg, 1975.
\newblock Graduate Texts in Mathematics, No. 21.

\bibitem{KLM}
Jarek K{\k e}dra, Assaf Libman, and Ben Martin.
\newblock On boundedness properties of groups.
\newblock {\em In preparation.}, https://arxiv.org/abs/1808.01815.

\bibitem{MR3822326}
Daniel~A. Marcus.
\newblock {\em Number fields}.
\newblock Universitext. Springer, Cham, 2018.
\newblock Second edition of [ MR0457396], With a foreword by Barry Mazur.

\bibitem{MR3892969}
Aleksander~V. Morgan, Andrei~S. Rapinchuk, and Balasubramanian Sury.
\newblock Bounded generation of {$\rm SL_2$} over rings of {$S$}-integers with
  infinitely many units.
\newblock {\em Algebra Number Theory}, 12(8):1949--1974, 2018.

\bibitem{MR2357719}
Dave~Witte Morris.
\newblock Bounded generation of {${\rm SL}(n,A)$} (after {D}. {C}arter, {G}.
  {K}eller, and {E}. {P}aige).
\newblock {\em New York J. Math.}, 13:383--421, 2007.

\bibitem{MR1697859}
J\"{u}rgen Neukirch.
\newblock {\em Algebraic number theory}, volume 322 of {\em Grundlehren der
  Mathematischen Wissenschaften [Fundamental Principles of Mathematical
  Sciences]}.
\newblock Springer-Verlag, Berlin, 1999.
\newblock Translated from the 1992 German original and with a note by Norbert
  Schappacher, With a foreword by G. Harder.

\bibitem{Nica}
B.~Nica.
\newblock On bounded elementary generation for sl n over polynomial rings.
\newblock {\em Isr. J. Math.}, pages 403--410, 2018.

\bibitem{MR2596772}
Wolfgang Rautenberg.
\newblock {\em A concise introduction to mathematical logic}.
\newblock Universitext. Springer, New York, third edition, 2010.
\newblock With a foreword by Lev Beklemishev.

\bibitem{MR3616493}
Robert Steinberg.
\newblock {\em Lectures on {C}hevalley groups}, volume~66 of {\em University
  Lecture Series}.
\newblock American Mathematical Society, Providence, RI, 2016.

\bibitem{MR1044049}
O.~I. Tavgen.
\newblock Bounded generability of {C}hevalley groups over rings of
  {$S$}-integer algebraic numbers.
\newblock {\em Izv. Akad. Nauk SSSR Ser. Mat.}, 54(1):97--122, 221--222, 1990.

\bibitem{Hessenbergform_explicit_strong_bound}
Alexander Trost.
\newblock Hessenberg forms for classical {C}hevalley groups and explicit strong
  boundedness.
\newblock In preparation.

\bibitem{MR2161255}
Peter V\'{a}mos.
\newblock 2-good rings.
\newblock {\em Q. J. Math.}, 56(3):417--430, 2005.

\bibitem{MR843808}
Leonid~N. Vaserstein.
\newblock On normal subgroups of {C}hevalley groups over commutative rings.
\newblock {\em Tohoku Math. J. (2)}, 38(2):219--230, 1986.

\bibitem{Vavilov_Subnormal}
N.~A. Vavilov.
\newblock A note on the subnormal structure of general linear groups.
\newblock {\em Math. Proc. Camb.Phil. Soc.}, 107:193--196, 1990.

\bibitem{MR2822515}
N.~A. Vavilov, A.~V. Smolenski\u{\i}, and B.~Sury.
\newblock Unitriangular factorizations of {C}hevalley groups.
\newblock {\em Zap. Nauchn. Sem. S.-Peterburg. Otdel. Mat. Inst. Steklov.
  (POMI)}, 388(Voprosy Teorii Predstavleni\u{\i} Algebr i Grupp. 21):17--47,
  309--310, 2011.

\bibitem{Hong}
Hong You.
\newblock Subgroups of classical groups normalized by relative elementary
  groups.
\newblock {\em Journal of Pure and Applied Algebra}, 216:1040--1051, 2012.

\end{thebibliography}
\bibliographystyle{plain}

\end{document}